\documentclass[11pt]{amsart}

\usepackage[T1]{fontenc}
\usepackage{lmodern}
\usepackage{amsmath,amssymb,amsthm,mathtools}
\usepackage[margin=1.08in]{geometry}
\usepackage{microtype}
\usepackage{xcolor}
\usepackage[colorlinks=true,linkcolor=blue!55!black,
  citecolor=blue!55!black,urlcolor=blue!55!black]{hyperref}

\allowdisplaybreaks 

\usepackage[capitalize]{cleveref}
\usepackage{relsize}
\usepackage{nicefrac}
\usepackage{mathtools}

\newtheorem{theorem}{Theorem}[section]
\newtheorem{proposition}[theorem]{Proposition}
\newtheorem{lemma}[theorem]{Lemma}
\newtheorem{corollary}[theorem]{Corollary}
\newtheorem{conjecture}[theorem]{Conjecture}

\theoremstyle{remark}
\newtheorem{remark}[theorem]{Remark}
\numberwithin{equation}{section}
\newcommand{\N}{\mathbb N}
\newcommand{\Z}{\mathbb Z}
\newcommand{\R}{\mathbb R}
\newcommand{\C}{\mathbb C}
\renewcommand{\P}{\mathbb{P}}

\newcommand{\Torus}{\mathbb T}
\newcommand{\one}{\mathbf 1}

\newcommand{\bh}{\mathbf h}
\newcommand{\ba}{\mathbf a}
\newcommand{\ve}{\varepsilon}

\newcommand{\set}[1]{\left\{#1\right\}}
\newcommand{\norm}[1]{\left\lVert#1\right\rVert}
\newcommand{\abs}[1]{\left|#1\right|}
\DeclareMathOperator{\Li}{Li}
\newcommand{\Nzero}{\mathbb{N}_0}
\newcommand{\1}{1}

\newcommand{\oh}{{\mathrm o}}

\newcommand{\E}{\raisebox{-0.4ex}{\scalebox{1.4}[1.3]{$\mathbb{E}$}}}
\newcommand{\Esub}[1]{\raisebox{-0.4ex}{\scalebox{1.4}[1.3]{$\mathbb{E}$}}\raisebox{-0.35ex}{\ensuremath{_{\scriptstyle #1}}}}

\newcommand{\BEu}[1]{\underset{#1}{\mathlarger{\mathlarger{\mathbb{E}}}^{~}}\,}

\author[B. Wang]{Biao Wang}
\address{School of Mathematics and Statistics, Yunnan University, Kunming, Yunnan 650500, China}
\email{bwang@ynu.edu.cn}

\author[S. Yi]{Shaoyun Yi}
\address{School of Mathematical Sciences, Xiamen University, Xiamen, Fujian 361005, China}
\email{yishaoyun926@xmu.edu.cn}

\date{\today}

\title[Averaged dynamical generalizations of Chowla's conjecture]
{Averaged dynamical generalizations of Chowla's conjecture}

\begin{document}
\begin{abstract}
In 1965, Chowla conjectured that the signs of the Liouville function become asymptotically uncorrelated at any fixed collection of distinct shifts. In 2015, Matom\"aki, Radziwi\l\l{} and Tao proved an averaged form of Chowla's conjecture. In 2022, Lichtman proved a variant of this conjecture  over primes on average. In the same year, Lichtman and Ter\"av\"ainen proved the Hardy--Littlewood--Chowla conjecture on average. In this article, motivated by the recent work of Bergelson and Richter on the dynamical generalizations of the prime number theorem, we will first establish the analogous dynamical generalizations of these three results related to Chowla's conjecture. In the proofs, we follow an approach of Qi and Zheng who established a variant of Bergelson and Richter's theorem over irreducible binary cubic forms. In 2016, Tao proved that the two-point logarithmic Chowla's conjecture holds. Recently, Charamaras and Richter generalized Tao's theorem to bounded arithmetic functions and proposed a conjecture that generalizes Chowla's conjecture to bounded multi-variable arithmetic functions. Building on their work, we will also prove a dynamical generalization of Tao's theorem and some variants for the composition of the sum-of-digits function with the prime Omega function.
 
\end{abstract}
\maketitle
\raggedbottom

\section{Introduction and main results}

Let $\Omega(n)$ denote the number of prime factors of a positive integer
$n$, counted with multiplicity, with $\Omega(0)=\Omega(1)=0$.
Throughout, $p$ denotes a prime, $\pi(N)=\#\{p\leq N\}$, and $\Lambda$
denotes the von Mangoldt function, extended by zero to nonpositive integers. The the prime number theorem (PNT)  states that $\pi(N)\sim {N}/{\log N}$. Write $\lambda(n)=(-1)^{\Omega(n)}$ for the Liouville function.
It is well-known (e.g., \cite{Landau1909}) that the PNT is equivalent to the assertion that
\begin{equation}\label{pnt_Liouville}
	\lim_{N\to\infty}\frac1N\sum_{n=1}^N\lambda(n)=0.
\end{equation}

Let $k\ge1$ be a natural number. In 1965, Chowla \cite{Chowla1965} gave the following conjecture on the $k$-point correlations of the Liouville function: 
		for any sequence of $k$ distinct non-negative integers  $h_1, \ldots, h_k$, we have
		\begin{equation}\label{Chowla_conjecture_eqn}
			\lim_{N\to\infty}\frac1N\sum_{n=1}^N \lambda(n+h_1)\ldots\lambda(n+h_k) = 0.
		\end{equation}
For the case $k=1$, Chowla's conjecture \eqref{Chowla_conjecture_eqn} is the same as the equivalent form \eqref{pnt_Liouville} of the PNT. For $k=2$,  \eqref{Chowla_conjecture_eqn} is recently proved in \cite{OAI}. For the averaged form of  \eqref{Chowla_conjecture_eqn}, Matom\"aki,  Radziwi\l\l{} and Tao \cite{MRT2015} proved in 2015 that for $k\ge1$ and  any $10\leq H\leq N$, we have
\begin{equation}\label{MRT2015}
	\sum_{1\leq h_1,\dots, h_k\leq H} \Big|\sum_{1\leq n\leqslant N} \lambda(n+h_1)\ldots\lambda(n+h_k)\Big| \ll_k \Bigg(\frac{\log\log H}{\log H}+ \frac1{\log^{1/3000}N}\Bigg)H^{k}N.
\end{equation}

 For the variant of  \eqref{pnt_Liouville} over shifted primes,  it is a folklore conjecture that
\begin{equation}\label{pnt_shifted_primes}
	\lim_{N\to\infty} \frac1{\pi(N)}\sum_{p\leq N}\lambda(p+1)=0,
\end{equation}
see \cite{Hildebrand1989}. Usually, one would expect that the numbers $p+1$ have the same multiplicative properties as the integers $n$, see \cite[pp 211]{Hildebrand1989}. For example, recently L\"u, Wang and Wu \cite{LvWangWu} proved that \eqref{pnt_shifted_primes} holds for the M\"obius function $\mu(p+2)$.  Inspired by this viewpoint, besides \eqref{pnt_shifted_primes}, one may ask  whether  the Chowla conjecture holds or not  along the subsequence of primes:
		\begin{equation}\label{Chowla_conjecture_shifted_primes}
			\lim_{N\to\infty}\frac{1}{\pi(N)} \sum_{p\leqslant N} \lambda(p+h_1)\cdots\lambda(p+h_k) = 0
		\end{equation}
for	distinct positive integers	$h_1, \ldots, h_k$.

In 2022, Lichtman \cite{Lichtman2022} proved an averaged form of \eqref{Chowla_conjecture_shifted_primes}: for $H<N$ and $\log H/\log\log N \to\infty$ as $N\to\infty$, we have
\begin{equation}\label{Lichtman2022}
	\sum_{1\leq h_1,\dots, h_k\leq H} \Big|\sum_{p\leqslant N} \lambda(p+h_1)\cdots\lambda(p+h_k)\Big| =o(H^{k}\pi(N)).
\end{equation}

In the same year,  Lichtman and Ter\"av\"ainen \cite{LT} proved the following hybrid form of the conjectures of Hardy–Littlewood and of Chowla: for any $k,\ell \ge1,\ve>0$, $(\log N)^{\ell +\ve} \leq H \leq \exp((\log N)^a)$ with $a=a(\ve,\ell)>0$ small enough, we have
\begin{equation}\label{LT2022}
\sum_{h_1\leq H} \Big|  \sum_{n\leq N}  \mu(n+h_1) \cdots \mu(n+h_k) \Lambda(n+a_1)\cdots \Lambda(n+a_\ell)\Big| \ll HN\frac{\log\log H}{\log H}.
\end{equation}

Let $(X,\nu,T)$ be a uniquely ergodic topological dynamical system:
$X$ is a compact metric space, $T:X\to X$ is continuous, and $\nu$
is its unique invariant Borel probability measure.
Let $C(X)$ denote the continuous complex-valued functions on $X$. In 2022, Bergelson and Richter \cite[Theorem A]{BergelsonRichter2022} proved a dynamical generalization of \eqref{pnt_Liouville} that
\begin{equation}\label{BergelsonRichter2022}
\lim_{N\to\infty} \frac1N\sum_{n\leq N}f(T^{\Omega(n)}x)=\int_X f\,d\nu
\end{equation}
for every $f\in C(X)$ and $x\in X$. For the rotation on two points, this recovers \eqref{pnt_Liouville}. 

Motivated by \eqref{BergelsonRichter2022}, in this article we will prove dynamical generalizations of
\eqref{MRT2015}, \eqref{Lichtman2022} and \eqref{LT2022} under certain ranges of $H$ first. Write  $\nu(f)=\int_X f\,d\nu.$ For a positive integer $H$, put $[H]=\{1,\ldots,H\}$ and
\[
 S_{\bh}(n)=\sum_{j=1}^k\Omega(n+h_j),\qquad
 \E_{\bh\in[H]^k}F(\bh)=\frac1{H^k}
 \sum_{1\leq h_1,\ldots,h_k\leq H}F(\bh)
\]
for any function $F:\N^k\to \C$.
Write $\bh'=(h_2,\ldots,h_k)$ and $\bh=(h_1,\bh')$, and put
\[
 \E_{h_1\in[H]}F(h_1,\bh')=\frac1H\sum_{h_1=1}^H F(h_1,\bh').
\]
In a single-shift average, $\bh'$ is fixed. Uniformity for
$\bh'\in[N]^{k-1}$ permits these shifts to depend on $N$ and $H$,
but not on the summation variable $h_1$.
For $k=1$, $[N]^0$ consists of the empty tuple.
Define continuous functions on $X$ by
\[
 A^{\rm int}_{N,\bh}f(x)=\frac1N\sum_{n\leq N}f(T^{S_{\bh}(n)}x),
 \qquad
 A^{\rm pr}_{N,\bh}f(x)=\frac1{\pi(N)}
 \sum_{p\leq N}f(T^{S_{\bh}(p)}x).
\]
For any $g\in C(X)$, define $\| g\|_{\infty} : = \sup_{x\in X} |g(x)|$.

\begin{theorem}\label{thm:integers}
Let $k\geq1$ be fixed. For every uniquely ergodic system $(X,\nu,T)$
and every $f\in C(X)$,
\begin{equation}\label{eq:integer-strong}
 \lim_{N\to\infty}\sup_{\log N\leq H\leq N}
 \sup_{\bh'\in[N]^{k-1}}
 \E_{h_1\in[H]}\norm{A^{\rm int}_{N,(h_1,\bh')}f-\nu(f)}_\infty=0.
\end{equation}
Consequently, for every integer-valued $H=H(N)$ with
$\log N\leq H\leq N$, every choice of
$\bh'=\bh'(N)\in[N]^{k-1}$, and every $x\in X$,
\begin{equation}\label{mainthm_dyn_Chowla_avg_eqn}
 \lim_{N\to\infty}\frac1{HN}
 \sum_{h_1=1}^{H}
 \abs{\sum_{n\leq N}f(T^{S_{(h_1,\bh')}(n)}x)}
 =\abs{\nu(f)}.
\end{equation}
The convergence is uniform in $x$, $H$, and $\bh'$ in the stated ranges.
\end{theorem}

\begin{theorem}\label{thm:primes}
Let $k\geq1$ be fixed. For every uniquely ergodic system $(X,\nu,T)$
and every $f\in C(X)$,
\begin{equation}\label{eq:prime-strong}
 \lim_{N\to\infty}\sup_{\exp(\sqrt{\log N})\leq H\leq N}
 \sup_{\bh'\in[N]^{k-1}}
 \E_{h_1\in[H]}\norm{A^{\rm pr}_{N,(h_1,\bh')}f-\nu(f)}_\infty=0.
\end{equation}
Consequently, for every integer-valued $H=H(N)$ with
$\exp(\sqrt{\log N})\leq H\leq N$, every choice of
$\bh'=\bh'(N)\in[N]^{k-1}$, and every $x\in X$,
\begin{equation}\label{mainthm_prime_chowla_avg_eqn}
 \lim_{N\to\infty}\frac1{H\pi(N)}
 \sum_{h_1=1}^{H}
 \abs{\sum_{p\leq N}f(T^{S_{(h_1,\bh')}(p)}x)}
 =\abs{\nu(f)}.
\end{equation}
The convergence is uniform in $x$, $H$, and $\bh'$ in the stated ranges.
\end{theorem}

We now introduce the fixed prime-tuple weights. Let $\ell\geq1$
and $\ba=(a_1,\ldots,a_\ell)\in\Z^\ell$ be fixed, with distinct
components, and put
\begin{equation}\label{eq:weight-mass}
 w_{\ba}(n)=\prod_{i=1}^{\ell}\Lambda(n+a_i),\qquad
 C_{\ba}(N)=\frac1N\sum_{n=1}^{N}w_{\ba}(n).
\end{equation}
In the notation of the preceding theorems, define
\[
 A^{\rm wt}_{N,\bh}f(x)=\frac1N\sum_{n=1}^{N}
          w_{\ba}(n)f(T^{S_{\bh}(n)}x).\]
The parameters $k,\ell,\ba$ remain fixed. The superscript $\mathrm{wt}$
always refers to this fixed weight. The $h_j$ need not be distinct
and may coincide with some of the $a_i$.

\begin{theorem}\label{thm:weighted}
For every uniquely ergodic system $(X,\nu,T)$ and every $f\in C(X)$,
\begin{equation}\label{eq:weighted-strong}
 \lim_{N\to\infty}\ \sup_{\exp(\sqrt{\log N})\leq H\leq N}
 \sup_{\bh'\in[N]^{k-1}}\E_{h_1\in[H]}
 \norm{A^{\rm wt}_{N,(h_1,\bh')}f-C_{\ba}(N)\nu(f)}_\infty=0.
\end{equation}

Consequently, for every integer-valued $H=H(N)$ satisfying
$\exp(\sqrt{\log N})\leq H\leq N$ and every
$\bh'=\bh'(N)\in[N]^{k-1}$, one has
\begin{equation}\label{eq:factorization}
 \frac1{HN}\sum_{h_1=1}^{H}\sum_{n=1}^{N}
 w_{\ba}(n)f(T^{S_{(h_1,\bh')}(n)}x)
 =\nu(f)C_{\ba}(N)+o(1),
\end{equation}
uniformly in $x\in X$ and in the permitted choices of $H$ and $\bh'$.
The corresponding average of absolute values satisfies
\begin{equation}\label{eq:absolute-factorization}
 \E_{h_1\in[H]}\abs{A^{\rm wt}_{N,(h_1,\bh')}f(x)}
 =C_{\ba}(N)\abs{\nu(f)}+o(1)
\end{equation}
with the same uniformity. If $\nu(f)=0$, both left-hand sides
tend to zero. If $\ell=1$, the limit in
\eqref{eq:factorization} is $\nu(f)$.
\end{theorem}

For a prime $p$, let
\[
 r_p(\ba)=\#\{a_1,\ldots,a_\ell\pmod p\},\qquad
 \mathfrak S(\ba)=\prod_p
    \left(1-\frac{r_p(\ba)}p\right)
    \left(1-\frac1p\right)^{-\ell}.
\]
The tuple is called \emph{admissible} if $r_p(\ba)<p$ for every
prime $p$.  If $\ba$ is not admissible, then $ \mathfrak S(\ba)=0$. The celebrated Hardy--Littlewood conjecture predicts  that
\begin{equation}\label{eq:HL}
\frac1N\sum_{n=1}^{N}\prod_{i=1}^{\ell}\Lambda(n+a_i) = \mathfrak S(\ba)+o(1).
\end{equation}

\begin{corollary}\label{cor:HL}
Assume the Hardy--Littlewood conjecture \eqref{eq:HL} holds. 
Then, uniformly for $\exp(\sqrt{\log N})\leq H\leq N$,
$\bh'\in[N]^{k-1}$, and $x\in X$,
\[
 \frac1{HN}\sum_{h_1=1}^{H}\sum_{n=1}^{N}
 w_{\ba}(n)f(T^{S_{(h_1,\bh')}(n)}x)
 \longrightarrow\mathfrak S(\ba)\nu(f).
\]
If $\ba$ is not admissible, this limit is unconditionally zero
for every choice of positive integers $H=H(N)$ and positive shifts
$\bh'=\bh'(N)$, without growth restrictions on either.
\end{corollary}

 The lower bounds in the range of $H$ in Theorems~\ref{thm:integers}, \ref{thm:primes}  and \ref{thm:weighted} are not optimal. One may use the method in this article to improve $H\ge \log N$ in \eqref{eq:integer-strong} to  $H\ge \exp((\log\log N)^{\frac12+\ve})$ for any $\ve>0$, and improve $H\ge \exp(\sqrt{\log N})$ in \eqref{eq:prime-strong} and \eqref{eq:weighted-strong} to  $H\ge \exp((\log\log N)^{4+\ve})$ for any $\ve>0$.

 In contrast to Theorems~\ref{thm:integers}, \ref{thm:primes}  and \ref{thm:weighted}, which are related to averages over $h_1,\dots, h_k$, our second kind of results are related to the logarithmic average of Chowla's conjecture \eqref{Chowla_conjecture_eqn}  over $n$ for $k=2$. Concerning the two-point Chowla conjecture, in 2016, Tao \cite{Tao2016} proved that 
\begin{equation}\label{Tao2016}
	\frac{1}{\log N}\sum_{n=1}^N \frac{\lambda(n)\lambda(n+1)}{n}=o_{N\to\infty}(1).
\end{equation}
The error term was improved to $O(\frac1{\sqrt{\log\log N}})$ by Helfgott and Radziwi\l\l{} \cite{HelfgottRadziwill2021} and later to $O(\frac1{(\log N)^c})$ by Pilatte \cite{Pilatte2026} for some constant $c>0$. Recently, Charamaras and Richter \cite{CharamarasRichter2025} generalized \eqref{Tao2016} to bounded functions over $\Omega(n)$ as follows: for any bounded functions $a,b\colon\N\to\C$, we have
    \begin{equation}\label{CharamarasRichter2025_thm_A}
        \frac1{\log N}\sum_{n=1}^N \frac{a(\Omega(n))b(\Omega(n+1))}{n}
        = \Big(\frac1N\sum_{n=1}^N a(\Omega(n))\Big)\Big(\frac1N\sum_{n=1}^N b(\Omega(n))\Big)
        + O\bigg(\frac{1}{\sqrt{\log\log N}}\bigg).
    \end{equation}
    The error term depends only on $\max\{\|a\|_\infty,\|b\|_\infty\}$. 
 Here the error term of \eqref{CharamarasRichter2025_thm_A} is from \cite[Appendix A]{CharamarasRichter2025}. In the same article, they posed a strong generalization of Chowla's conjecture.
\begin{conjecture}[{\cite[Conjecture 1.10]{CharamarasRichter2025}}]
\label{conj_functional_chowla_str}
For any $k\in\N$ and any bounded $a\colon\N^k\to\C$, we have
$$
\frac1N\sum_{n=1}^N a\big(\Omega(n),\Omega(n+1),\dots,\Omega(n+k-1)\big)
= \frac{1}{N^{k}} \sum_{1\leq n_1,\dots, n_k\leq N}  a\big(\Omega(n_1),\dots,\Omega(n_k)\big) + \oh_{N\to\infty}(1).
$$
\end{conjecture}

In the following, we will show that Conjecture~\ref{conj_functional_chowla_str} implies a dynamical generalization of Chowla's conjecture \eqref{Chowla_conjecture_eqn} for $h_1=0,h_2=1, \dots, h_k=k-1$.

\begin{theorem}\label{Chowla_conjecture_BR_form_conditional_result}
Let $(X,\nu,T)$ be a unique ergodic system.	If Conjecture~\ref{conj_functional_chowla_str} is true, then 
\begin{equation}\label{Chowla_conjecture_BR_form}
		\lim_{N\to\infty} \frac1N\sum_{n=1}^N f(T^{\Omega(n)+ \Omega(n+1) +\cdots + \Omega(n+k-1)}x) =\int_X f\,d\nu
\end{equation}
holds for any $x\in X$ and $f\in C(X)$.
	
\end{theorem}

Moreover, we prove that the two-point logarithmic averaged Conjecture~\ref{conj_functional_chowla_str} holds. It is a generalization of \eqref{CharamarasRichter2025_thm_A} with a weaker error term.

\begin{theorem}\label{Charamaras_Richter_conj_log_avg}
	For every bounded function
$a\colon\N^2\to\C$, we have
\begin{equation}\label{Charamaras_Richter_conj_log_avg_eqn}
	\frac1{\log N}\sum_{n=1}^{N}
       \frac{a(\Omega(n),\Omega(n+1))}{n} = \frac1{N^2}\sum_{n_1= 1}^N\sum_{n_2= 1}^N
       a(\Omega(n_1),\Omega(n_2)) + O\Big(\frac{1}{(\log\log N)^{1/5}}\Big).
\end{equation}
\end{theorem}

As a consequence of Theorem~\ref{Charamaras_Richter_conj_log_avg}, we get the following dynamical generalization of Tao's theorem \eqref{Tao2016}. Equation \eqref{Tao2016} can be recovered from \eqref{Tao2016_dyn_eqn} by taking $(X,\nu, T)$ to be a rotation on two points.

\begin{theorem}\label{Tao2016_dyn}
Let $(X,\nu, T)$ be a uniquely ergodic topological dynamical system. Then we have
\begin{equation}\label{Tao2016_dyn_eqn}
	\lim_{N\to\infty} \frac{1}{\log N}  \sum_{n=1}^{N} \frac{f(T^{\Omega(n)+\Omega(n+1)}x)}n=\int_X f\,d\nu
\end{equation}
for any $x\in X$ and $f\in C(X)$. 	
\end{theorem}

At the end of this section, we introduce two results on the sum-of-digits function. Fix an integer $q\geq 2$.  If
\[
 n=\sum_{j\geq 0}\varepsilon_j(n)q^j,
 \qquad  \varepsilon_j(n)\in\set{0,1,\dots, q-1},
\]
we define
\[
 s_q(n):=\sum_{j\geq 0}\varepsilon_j(n),
 \qquad n\in\N
\]
to be the sum of digits of $n$ in base $q$,
with $s_q(0)=0$. Let $m\ge2$. In \cite{Gelfond1968}, Gel${}'$fond proved that the set of $n$ for which $s_q(n) \equiv r\pmod{m}$ has asymptotic density $1/m$ for $(m,q-1)=1$. In \cite[Theorem B]{Saavedra_Araya2026}, Saavedra-Araya showed that the sufficient
condition  $(m,q-1)=1$ in Gel${}'$fond's theorem is not necessary.  In \cite{BergelsonReillyRichter2026}, Bergelson et al. proved that \eqref{BergelsonRichter2022} holds if one replaces $\Omega(n)$ by $s_q(n)$. This will also imply Saavedra-Araya's result by taking the rotation on $m$ points. In the following, we will show a uniform  convergence for the ergodic averages along $s_q(n)$.

\begin{theorem}\label{thm:main_sq}
In a uniquely ergodic system $(X,\nu, T)$,  for every $f\in C(X)$, we have
\[
 \lim_{N\to\infty}\sup_{x\in X}
 \Big|
  \frac1N\sum_{n=1}^{N}f\bigl(T^{s_q(n)}x\bigr)
  -\int_X f\,d\nu
 \Big|=0.
\]
\end{theorem}

Moreover, we will prove that Theorem~\ref{thm:main_sq} also holds if we replace $s_q(n)$ by $s_q(a(n))$ for arithmetic functions $a:\N\to\N$ of shift invariance. 

\begin{theorem}\label{thm:one-point}
Let $(X,\nu, T)$ be uniquely ergodic. Let $a:\N\to\N$ be an arithmetic function. Let
$$\overline{\eta}_N(\ell):=\frac1N |\set{n\leq N: a(n)=\ell}|.$$
 If
\begin{equation}\label{thm_one_point_cond}
	 \sum_{\ell\ge0} |\overline{\eta}_N(\ell+1)-\overline{\eta}_N(\ell)|=o_{N\to\infty}(1),
\end{equation}
then we have
\begin{equation}\label{thm_one_point_eqn}
	 \lim_{N\to\infty}\sup_{x\in X}
 \Big|
  \frac1N\sum_{n=1}^{N}f\bigl(T^{s_q(a(n))}x\bigr)
  -\int_Xf\,d\nu
 \Big|=0
\end{equation}
for every $f\in C(X)$. 
\end{theorem}

For example, \eqref{thm_one_point_eqn} holds for $a(n)=\Omega(n)$ by Lemma~\ref{lem:translation}. 
Let $m\geq2$. In \cite[Corollary 1.7]{BergelsonRichter2022}, Bergelson and Richter showed an analogue of Gel${}'$fond's theorem for $s_q(\Omega(n))$. That is, if $(m,q-1)=1$ then the set of $n$ for which $s_q(\Omega(n)) \equiv r\pmod{m}$ has asymptotic density $1/m$.  Taking $a(n)=\Omega(n)$ and $(X,\nu, T)$ to be a rotation on $m$ points in Theorem~\ref{thm:one-point}, we obtain that the sufficient
condition  $(m,q-1)=1$ is not necessary. 

\begin{corollary}\label{thm_Gelfond1968}
Let $q\ge2$. Then we have
\begin{equation}\label{thm_Gelfond1968_eqn}
	\lim_{N\to\infty} \frac1N|\set{n\leq N : s_q(\Omega(n))  \equiv r \pmod{m}} | =\frac{1}{m}
\end{equation}
for any $m\in \N$ and $r\in\set{0,1,\dots,m-1}$.	
\end{corollary}

	Using the arguments in the proof of Theorem~\ref{thm:one-point}, one can also prove that both \cite[Theorem D]{DLMS2024} and \cite[Theorem B]{QiZheng2026} hold if we replace $\Omega(m^2+n^2)$ and $\Omega(|P(m,n)|)$ by $s_q(\Omega(m^2+n^2))$ and $s_q(\Omega(|P(m,n)|))$ respectively, where $P(m,n)$ is an irreducible cubic form. The main reason is that we have \cite[(14)]{CespedesDonoso2026} and \cite[Theorem 1]{QiZheng2026}, which can be viewed as counterparts of \eqref{thm_one_point_cond} for $\Omega(m^2+n^2)$ and $\Omega(|P(m,n)|)$. 
In the following, we show that Theorem~\ref{Tao2016_dyn} holds as well, if $\Omega(n)$ is replaced by $s_q(\Omega(n))$.

\begin{theorem}\label{thm:two-point}
For every $f\in C(X)$, we have
\[
 \lim_{N\to\infty}\sup_{x\in X}
 \Big|
  \frac1{\log N}\sum_{n=1}^{N}
  \frac{f\bigl(T^{s_q(\Omega(n))+s_q(\Omega(n+1))}x\bigr)}n
  -\int_Xf\,d\nu
 \Big|=0.
\]
\end{theorem}

The above results are closely related to Bergelson and Richter's theorem \eqref{BergelsonRichter2022}.  Currently, to the best of the author’s knowledge, there are at least three different approaches to prove \eqref{BergelsonRichter2022}. In \cite{BergelsonRichter2022}, Bergelson and Richter used an elementary and combinatorial method to obtain \eqref{BergelsonRichter2022}. This method was used by the author \cite{Wang2022ffa} to establish an analogue of \eqref{BergelsonRichter2022} over finite fields. It was also used by Loyd \cite{Loyd2023} to establish a disjoint form of  Bergelson and Richter's theorem with the Erd\H{o}s-Kac theorem. Donoso, Le, Moreira and Sun \cite{DLMS2024} applied Bergelson and Richter's method to establish a variant of  \eqref{BergelsonRichter2022} for $\Omega(m^2+n^2)$. The author \cite{Wang2025procA} also gave a variant of Donoso et al.'s result over co-prime integers. The second method to prove \cref{BergelsonRichter2022} was proposed by Kanigowski and Radziwi\l\l{}, see \cite[Remark 1.3]{BergelsonRichter2022}. It relies on the work of Erd\H{o}s \cite{Erdos1948} on the local Gaussian behavior of the $\Omega(n)$. It was used by C\'espedes and Donoso \cite{CespedesDonoso2026} to establish a generalization of \eqref{BergelsonRichter2022} over number fields. The way of applying the local Gaussian law was also used by Bergelson, Reilly and Richter \cite{BergelsonReillyRichter2026} to establish variants of \eqref{BergelsonRichter2022} for $s_q(n)$ and $s_2(p_n)$, where $s_q(n)$ denotes the $q$-ary sum-of-digits function of a non-negative integer $n$ for any base $q\ge2$, and $p_n$ denotes the $n$th prime. The third method was proposed by Qi and Zheng \cite{QiZheng2026} recently. It relies on the estimates of exponential sums along $\Omega(n)$ and the Fourier analysis. They used it to   establish a variant of  \eqref{BergelsonRichter2022} for $\Omega(|P(m,n)|)$, where $P(m,n)$ is an irreducible cubic form. In this article, to prove Theorems~\ref{thm:integers}, \ref{thm:primes}  and \ref{thm:weighted}, we will use the third method. Using this method, one can also verify that Theorems~\ref{thm:integers}, \ref{thm:primes}  and \ref{thm:weighted} hold for $\omega(n)$, where $\omega(n)$ denotes the number of distinct prime factors of $n$.

The article is organized as follows. In Section~\ref{sec_preliminaries}, we will introduce some background on unique ergodicity, some applications of Bergelson and Richter's theorem, and some results on the Erd\H{o}s-Kac theorem. To apply the method of Qi and Zheng, in Section~\ref{sec:distance}, we will prove a uniform pretentious-distance
estimate. Then in Section~\ref{sec:correlations}, we will establish uniform short-shift
correlation bounds related to $\Omega(n+h_1),\dots, \Omega(n+h_k)$. These bounds will be used to prove the averaged
shift invariance on their local densities in Section~\ref{sec:invariance}. In Section~\ref{sec:dynamics} we will complete the proof of Theorems~\ref{thm:integers}, \ref{thm:primes} and \ref{thm:weighted}. In Sections~\ref{sec_Chowla_BR_proof} and \ref{sec_Tao2016_dyn}, building on the results of Charamaras and Richter in \cite{CharamarasRichter2025}, we will prove Theorems~\ref{Chowla_conjecture_BR_form_conditional_result}-\ref{Tao2016_dyn} respectively. In Sections~\ref{sec_sq_proofs} and \ref{sec:second-proof}, we will use some additive properties of the sum-of-digits function to prove Theorems~\ref{thm:main_sq} and  \ref{thm:two-point} respectively. Finally, some problems are proposed in the last section.

\section{Preliminaries}
\label{sec_preliminaries}

\subsection{Unique Ergodicity}\label{sec_Unique_Ergodicity_bg}

Let $X$ be a compact metric space, and let $T:X\to X$ be a continuous map. Then the pair $(X,T)$ is called a \textit{topological dynamical system}. A Borel probability measure $\nu$ on $X$ is called \textit{T-invariant} if $\nu(T^{-1}A) = \nu(A)$ for all measurable subsets $A\subset X$. Due to the Bogolyubov-Krylov theorem (e.g., \cite[Corollary 6.9.1]{Walters1982}), every topological dynamical system has at least one $T$-invariant measure. If a topological system $(X, T)$ admits only one $T$-invariant measure $\nu$, then $(X, T)$ is called \textit{uniquely ergodic}. The following are three classical uniquely ergodic systems. 

\begin{enumerate}
	\item (Rotation on two points). Let $X_1=\{0,1\}$, $\nu(\set{0})=\nu(\set{1})=1/2$ and $T_1: x\mapsto x+1 \mod{2}$. Then $(X_1,\nu, T_1)$ is called a \textit{rotation on two points}. It is the simplest uniquely ergodic topological dynamical system. 
	\item (Rotation on $m$ points). For any $m\ge2$, let $X_2=\{0,1,\dots, m-1\}$, $\nu(\set{0})=\nu(\set{1})=\cdots=\nu(\set{m-1})=1/m$ and $T_2: x\mapsto x+1 \mod{m}$. Then $(X_2,\nu, T_2)$ is called a \textit{rotation on $m$ points}.  This system is  uniquely ergodic as well.
	\item (Irrational circle rotation). For any irrational number $\alpha\in \R$, let $X_3 = [0,1]$, endowed with the Lebesgue measure $\nu$, and define $T_\alpha$ by $T_\alpha x=x+\alpha \mod{1}$. Then the system $(X_3, \nu, T_\alpha)$ is called a \textit{irrational circle rotation}. It is uniquely ergodic, too. 
\end{enumerate}

It is well-known (e.g., \cite[Theorem 6.19]{Walters1982}) that $(X, \nu, T)$ is uniquely ergodic if and only if 
\begin{equation}\label{eqn_ergodicity}
	\lim_{N \to \infty} \frac1N\sum_{n=1}^N  f(T^n x) = \int_X f \, d\nu
\end{equation}
holds for all $x \in X$ and $f\in C(X)$. Moreover, this convergence is uniform for all $x \in X$, see \cite[Theorem 6.19(i)]{Walters1982} or \cite[Theorem 4.10(4)]{EinsiedlerWard2011}. That is, for any $f\in C(X)$ we have
\begin{equation}\label{uniform_ergodicity}
	 \lim_{N\to\infty}\sup_{x\in X}
 \Big|
  \frac1N\sum_{n=1}^{N}f\bigl(T^nx\bigr)
  -\int_Xf\,d\nu
 \Big|=0.
\end{equation}

\subsection{Applications of Bergelson and Richter's theorem}
\label{sec_BR_applications}

Bergelson and Richter's Theorem~\eqref{BergelsonRichter2022} tells us that Eq. \eqref{eqn_ergodicity} also holds if the orbit $\set{T^nx: n\in\N}$ is replaced by the orbit $\set{T^{\Omega(n)}x: n\in\N}$ along $\Omega(n)$ for every point $x\in X$. In \eqref{BergelsonRichter2022}, if we take three examples listed in \S\ref{sec_Unique_Ergodicity_bg}, we will obtain three classical results in analytic number theory listed below.

\begin{enumerate}
	\item In  \eqref{BergelsonRichter2022}, taking $(X,T)$ as a rotation on two points $(X_1,T_1)$ gives the equivalent form \eqref{pnt_Liouville} of the PNT. It is equivalent to saying that $\Omega(n)$ distributes evenly over residue classes modulo 2.
	\item In  \eqref{BergelsonRichter2022},  taking $(X,T)$ as  a rotation on  $m$ points $(X_2,T_2)$ ($m\ge2$)  implies that $\Omega(n)$ distributes evenly over all residue classes modulo $m$. This is a classical result of  Pillai and Selberg \cite{Pillai1940, Selberg1939}.
	\item In  \eqref{BergelsonRichter2022}, taking $(X,T)$ as irrational circle rotations $(X_3,T_\alpha)$ implies that the sequence $\{\Omega(n)\alpha\}_{n\in \N}$ is uniformly distributed mod 1 for any irrational number $\alpha$. This statement is mentioned by Erd\H{o}s in \cite{Erdos1946} without proof and later proved by Delange \cite{Delange1958}.
\end{enumerate}

Thus, equation \eqref{BergelsonRichter2022}  provides a simultaneous generalization of these number theoretic results from a dynamical point of view. We refer readers to their work \cite{BergelsonRichter2022} for more interesting applications.

\subsection{The Erd\H{o}s-Kac theorem}

The well-known Erd\H{o}s-Kac Theorem \cite{ErdosKac1940} says that the number of prime divisors $\Omega(n)$ of $n$ satisfies the following Gaussian distribution:

\begin{theorem}[Erd\H{o}s-Kac Theorem]
\label{thm_EK}
    Let $C_c(\R)$ denote the set of compactly supported continuous functions on $\R$, then we have 
	\begin{equation}\label{eqn_EK_origin}
	\lim_{N \to \infty} \frac1N\sum_{n=1}^N F \Big( \frac{\Omega(n) - \log \log N }{\sqrt{\log \log N}} \Big)=\frac{1}{\sqrt{2\pi}} \int_{-\infty}^{\infty} F(t) e^{-t^2/2} \, dt
\end{equation}
for any $F\in C_c(\R)$.
\end{theorem}

One way to prove Theorem~\ref{thm_EK} is to use the Sathe-Selberg  theorem, see \cite[Theorem 7.21]{MontgomeryVaughan2007}. In the proof, the following uniform form of the local Gaussian
law from
\cite[p. 236]{MontgomeryVaughan2007} will be useful in the proof of our main results.

\begin{theorem}[Local Erd\H{o}s-Kac Theorem]\label{thm_strong-local-EK}
Let $L_N=\log\log N$ for $N\ge3$. We have
\begin{equation}\label{thm_strong-local-EK_eqn}
	\frac1N|\set{n\leq N: \Omega(n)=\ell}|= \frac{\exp(-\frac{(\ell-L_N)^2}{2L_N})}{\sqrt{2\pi L_N}} \Big(1+O\!\Big(
       \frac{1}{\sqrt{L_N}}
       +\frac{|\ell-L_N|^3}{L_N^2}
     \Big) \Big)
\end{equation}
uniformly for all positive integers
$\ell\leq \frac{3}{2}L_N$.
\end{theorem}

For any $\mu\in \R$ and $\sigma>0$, let $f(x;\mu,\sigma)$ be the Gaussian probability density function  given by
\begin{equation}
\label{eqn_Gaussian_prob_dens_fctn}
f(x;\mu,\sigma)= \frac{1}{\sqrt{2\pi}\sigma}\,e^{-\frac{1}{2}\left(\frac{x-\mu}{\sigma}\right)^2}.
\end{equation}
We also define
\[
  G_t(\ell)
  :=f(\ell;t,\sqrt t)
  =\frac{1}{\sqrt{2\pi t}}
    \exp\!\left(-\frac{(\ell-t)^2}{2t}\right),
  \qquad t>0.
\]

The following estimate will be useful in the applications of Theorem~\ref{thm_strong-local-EK}.

\begin{lemma}\label{lem_lattice_moments}
For every fixed nonnegative real number $r\geq 0$, we have
\begin{equation}\label{eq:lattice-moments}
\sum_{\ell\in\mathbb Z}G_t(\ell)\lvert \ell-t\rvert^r
\ll_r t^{r/2}
\end{equation}
for all $t\ge1$.
\end{lemma}

\begin{proof}
Fix $r\geq 0$ and $t\geq 1$. For each integer $k\geq 0$, set
\[
A_k
=\left\{\ell\in\mathbb Z:
  k\sqrt{t}\leq \lvert \ell-t\rvert<(k+1)\sqrt{t}\right\}.
\]
Then
\(
|A_k|\leq 2\sqrt{t}+2\leq 4\sqrt{t}.
\)
If $\ell\in A_k$, then
\[
\lvert \ell-t\rvert^r
\leq (k+1)^r t^{r/2}
\]
and
\[
\exp\!\left(-\frac{(\ell-t)^2}{2t}\right)
\leq e^{-k^2/2}.
\]
Consequently,
\begin{align*}
\sum_{\ell\in\mathbb Z}G_t(\ell)\lvert \ell-t\rvert^r
&\leq \frac{1}{\sqrt{2\pi t}}
  \sum_{k=0}^{\infty}
  |A_k|\,(k+1)^r t^{r/2}e^{-k^2/2} \\
&\leq \frac{4}{\sqrt{2\pi}}\,t^{r/2}
  \sum_{k=0}^{\infty}(k+1)^r e^{-k^2/2}.
\end{align*}
The series on the right converges and depends only on $r$. This proves
\eqref{eq:lattice-moments}.
\end{proof}

\section{A uniform pretentious-distance estimate}\label{sec:distance}

We use $e(t)=\exp(2\pi it)$ and identify $\mathbb T=\mathbb R/\mathbb Z$
with $[-1/2,1/2]$ when convenient. Integration on $\mathbb T$ uses
usual Lebesgue measure. 
For $1$-bounded multiplicative functions $f,g$, put
\[
 \mathbb D(f,g;Y)^2
 =\sum_{p\leq Y}\frac{1-\Re(f(p)\overline{g(p)})}{p},
 \qquad
 \mathcal M(f;Y,Q)
 =\inf_{\substack{q\leq Q,\ \chi\bmod q\\ |t|\leq Y}}
 \mathbb D(f,\chi n^{it};Y)^2.
\]
Here $\chi n^{it}$ denotes $n\mapsto\chi(n)n^{it}$.
We have the triangle inequality for $\mathbb D$:
\begin{equation}\label{GS_triangle_inequality}
\mathbb D(1,fg;Y)\leq \mathbb D(1,f;Y) + \mathbb D(1,g;Y),
\end{equation}
see \cite{GS}.

\begin{lemma}\label{lem:oscillatory}
If $0<a\leq b$ and $t\in\mathbb R$, then
\[
 I(a,b;t)=\int_a^b\frac{e^{-itv}}v\,dv
\]
satisfies
\[
 |\Im I(a,b;t)|\leq3,\qquad
 -2\leq\Re I(a,b;t)\leq\log(b/a).
\]
\end{lemma}

\begin{proof}
If $t=0$, then $I(a,b;0)=\log(b/a)$ is real and nonnegative,
so all three assertions follow immediately.
Suppose that $t\neq0$, and write
\[
 A=|t|a,\qquad B=|t|b,\qquad \sigma=\operatorname{sgn}(t).
\]
The substitution $s=|t|v$ gives
\[
 I(a,b;t)=\int_A^B\frac{\cos s}{s}\,ds
          -i\sigma\int_A^B\frac{\sin s}{s}\,ds.
\]
We estimate the two real integrals uniformly in $0<A\leq B$.

For $1\leq u\leq v$, integration by parts gives
\[
 \int_u^v\frac{\sin s}{s}\,ds
 =\left[-\frac{\cos s}{s}\right]_u^v
   -\int_u^v\frac{\cos s}{s^2}\,ds.
\]
By the triangle inequality,
we have
\[
 \left|\int_u^v\frac{\sin s}{s}\,ds\right|
 \leq\frac1u+\frac1v+\int_u^v\frac{ds}{s^2}
 =\frac2u\leq2.
\]
Likewise,
\[
 \int_u^v\frac{\cos s}{s}\,ds
 =\left[\frac{\sin s}{s}\right]_u^v
   +\int_u^v\frac{\sin s}{s^2}\,ds,
\]
so
\[
 \left|\int_u^v\frac{\cos s}{s}\,ds\right|
 \leq\frac1u+\frac1v+\int_u^v\frac{ds}{s^2}
 =\frac2u\leq2.
\]

On $(0,1]$, we have $|\sin s/s|\leq1$ and $\cos s/s\geq0$.
The sine integral over any subinterval of $(0,1]$  has
absolute value at most $1$, while the cosine integral over that
subinterval is nonnegative. Split $[A,B]$ at $1$, omitting either portion if it is empty.
The preceding estimates give
\[
 \left|\int_A^B\frac{\sin s}{s}\,ds\right|\leq1+2=3,
 \qquad
 \int_A^B\frac{\cos s}{s}\,ds\geq-2.
\]
Finally, the pointwise inequality $\cos(tv)\leq1$ yields
\[
 \Re I(a,b;t)=\int_a^b\frac{\cos(tv)}v\,dv
 \leq\int_a^b\frac{dv}{v}=\log(b/a).
\]
These bounds are independent of $a,b,t$, as required.
\end{proof}

We record explicitly the standard character estimate used below.
Let
\[
 V_t=\exp\left((\log(3+|t|))^{2/3}
                       (\log\log(3+|t|))^{1/3}\right).
\]
For a character $\psi$ modulo $q$, by \cite[(4.4)]{Koukoulopoulos2013},  the conditions
\begin{equation}\label{eq:rough-conditions}
 2\le y\le Y,\qquad y\ge qV_t,\qquad |t|\ge1/\log y
\end{equation}
imply
\begin{equation}\label{eq:rough-character}
 \Re\sum_{y<p\le Y}\frac{\psi(p)p^{-it}}p=O(1),
\end{equation}
uniformly in these parameters, including principal characters.

For $n\in\N$, define $u_\theta(n)=e(\theta\Omega(n))$.
The identity $\Omega(mn)=\Omega(m)+\Omega(n)$ gives
$u_\theta(mn)=u_\theta(m)u_\theta(n)$ for all positive integers
$m,n$. Thus $u_\theta$ is completely multiplicative and has
absolute value $1$ on $\N$. We extend $u_\theta$ by zero to
nonpositive integers when it occurs in sums over $\Z$.

\begin{lemma}\label{lem:full-distance}
There are absolute constants $c,C>0$ such that, for all sufficiently
large $Y$ and every $|\theta|\leq1/2$,
\begin{equation}\label{eq:full-distance}
 \inf_{\substack{q\leq(\log Y)^{1/125},\ \chi\bmod q\\ |t|\leq Y}}
 \mathbb D(u_\theta,\chi n^{it};Y)^2
 \geq c\theta^2\log\log Y-C.
\end{equation}
\end{lemma}

\begin{proof}
Put $\mathcal L=\log\log Y$ and $z=e(\theta)$.
Fix a modulus $q\leq(\log Y)^{1/125}$, a character $\chi$
modulo $q$, and a real $t$ with $|t|\leq Y$. We consider two cases: $|t|\leq1$ and $1<|t|\leq Y$.

Suppose $|t|\leq1$.
Set $W=\exp((\log Y)^{1/2})$ and write
\[
 A_\chi(v)=\sum_{p\leq v}\overline{\chi(p)},\qquad
 \kappa_\chi=\mathbf1_{\chi=\chi_0}.
\]
The Siegel--Walfisz theorem in character form
\cite[Chapter 22]{Davenport} gives
\[
 A_\chi(v)=\kappa_\chi\Li(v)+E_\chi(v),\qquad
 |E_\chi(v)|\ll_{A,B}\frac{v}{(\log v)^A}
 \quad(q\leq(\log v)^B).
\]
For $v\geq W$,
\[
 q\leq(\log Y)^{1/125}=(\log W)^{2/125}
 \leq(\log v)^{2/125}.
\]
Thus we may fix, for example, $A=3$ and $B=2/125$; the
error estimate is then uniform in all parameters under consideration.
By definition,
\begin{align*}
 \mathbb D(u_\theta,\chi n^{it};Y)^2
 =\sum_{p\leq Y}\frac1p
      -\Re\left(z\sum_{p\leq Y}
                         \frac{\overline{\chi(p)}}{p^{1+it}}\right)=\Re\sum_{p\leq Y}\frac{1-z\overline{\chi(p)}p^{-it}}p.
\end{align*}

Let $w_t(v)=v^{-1-it}$, so that
$w_t'(v)=-(1+it)v^{-2-it}$.
Partial summation over the half-open interval $(W,Y]$ gives
\begin{align*}
 \sum_{W<p\leq Y}\overline{\chi(p)}w_t(p)
 &=A_\chi(Y)w_t(Y)-A_\chi(W)w_t(W)
       -\int_W^Y A_\chi(v)w_t'(v)\,dv\\
 &=\kappa_\chi\int_W^Y\frac{v^{-1-it}}{\log v}\,dv
       +\mathcal E,
\end{align*}
where
\[
 \mathcal E=E_\chi(Y)w_t(Y)-E_\chi(W)w_t(W)
                 -\int_W^Y E_\chi(v)w_t'(v)\,dv.
\]
In the main term we used $\Li'(v)=1/\log v$ and integrated
by parts in the opposite direction. The error satisfies
\begin{align*}
 |\mathcal E|
 &\ll (\log Y)^{-3}+(\log W)^{-3}
          +(1+|t|)\int_W^Y\frac{dv}{v(\log v)^3}\\
 &\ll (\log W)^{-3}
          +(1+|t|)(\log W)^{-2}\ll1,
\end{align*}
uniformly for $|t|\leq1$.
Substituting $v=e^w$ in the main integral therefore yields
\begin{equation}\label{eq:small-height-primes}
 \sum_{W<p\leq Y}\frac{\overline{\chi(p)}}{p^{1+it}}
 =\kappa_\chi\int_{\log W}^{\log Y}\frac{e^{-itw}}w\,dw+O(1).
\end{equation}

Let $I$ denote the integral in \eqref{eq:small-height-primes}.
Since $\log W=(\log Y)^{1/2}$, Lemma~\ref{lem:oscillatory}
gives
\[
 -2\leq\Re I\leq\log\frac{\log Y}{\log W}
       =\frac{\mathcal L}{2},\qquad |\Im I|\leq3.
\]
Write $z=a+ib$, where $a=\cos(2\pi\theta)$ and $|b|\leq1$.
If $a\geq0$, then $\Re(zI)=a\Re I-b\Im I\leq a\mathcal L/2+3$.
If $a<0$, then $a\Re I\leq-2a\leq2$, so $\Re(zI)\leq5$.
Both cases imply
\[
 \Re(zI)\leq\max(a,0)\frac{\mathcal L}{2}+5.
\]
For a nonprincipal character the main term in
\eqref{eq:small-height-primes} is zero, so the same upper bound
holds for the real part of its prime sum multiplied by $z$.
Also, Mertens' formula gives
\[
 \sum_{W<p\leq Y}\frac1p
 =\log\log Y-\log\log W+O(1)=\frac{\mathcal L}{2}+O(1).
\]
Every term in $\mathbb D^2$ is nonnegative, since
$|z\overline{\chi(p)}p^{-it}|\leq1$. We may therefore restrict
the sum to $p>W$ and obtain
\begin{align}
 \mathbb D(u_\theta,\chi n^{it};Y)^2
 &\geq\sum_{W<p\leq Y}\frac1p
      -\Re\left(z\sum_{W<p\leq Y}
                         \frac{\overline{\chi(p)}}{p^{1+it}}\right)\notag\\
 &\geq\frac{\mathcal L}{2}\bigl(1-\max(\cos(2\pi\theta),0)\bigr)-O(1)\notag\\
 &\geq2\theta^2\mathcal L-O(1).
 \label{eq:small-height-distance}
\end{align}
To check the last step, when $|\theta|\leq1/4$ use
$\sin(\pi|\theta|)\geq2|\theta|$ to get
$1-\cos(2\pi\theta)=2\sin^2(\pi\theta)\geq8\theta^2$.
When $1/4\leq|\theta|\leq1/2$, we have $\cos(2\pi\theta) \leq 0$, the maximum with zero is zero,
and $1\geq4\theta^2$. Thus
$1-\max(\cos(2\pi\theta),0)\geq4\theta^2$ for all $|\theta|\leq1/2$.

Suppose now that $1\leq|t|\leq Y$, and set
$y=\exp((\log Y)^{3/4})$.
We shall first consider any character $\psi$ modulo $q$ and
any $s$ with $1\leq|s|\leq100Y$.
For such $q,s$,
\[
 \log(qV_s)
 \leq\frac1{125}\log\log Y
     +O\bigl((\log Y)^{2/3}(\log\log Y)^{1/3}\bigr)
 =o((\log Y)^{3/4}).
\]
Hence $qV_s\leq y$ for all sufficiently large $Y$, uniformly
in $q,s$. Moreover, $|s|\geq1\geq1/\log y$ and $2\leq y\leq Y$.
 Applying \eqref{eq:rough-character} to $\overline\psi$ and
using Mertens' formula once again give
\begin{align}
 \mathbb D(1,\psi n^{is};Y)^2
 &\geq\sum_{y<p\leq Y}\frac1p
       -\Re\sum_{y<p\leq Y}\frac{\overline{\psi(p)}p^{-is}}p\notag\\
 &=\log\log Y-\log\log y+O(1)
  =\frac{\mathcal L}{4}+O(1).
 \label{eq:character-repulsion}
\end{align}
This implies that
\[
 \mathbb D(1,\psi n^{is};Y) \geq \frac25\sqrt{\mathcal L},
\]
 once $Y$ is sufficiently large.

By Dirichlet's approximation theorem, there is  an integer
$a$ such that
\[
 1\leq a\leq100,\qquad
 \|a\theta\|_{\mathbb R/\mathbb Z}\leq\frac1{100}.
\]
Here \(\|x\|_{\mathbb R/\mathbb Z}
:=\min_{m\in\mathbb Z}|x-m|\).
Since $u_\theta(p)^a=e(a\theta)$ at every prime,
\begin{align*}
 \mathbb D(u_\theta^a,1;Y)^2
 &=(1-\cos(2\pi a\theta))\sum_{p\leq Y}\frac1p \leq\frac{2\pi^2}{100^2}(\mathcal L+O(1)).
\end{align*}
Here $1-\cos v\leq v^2/2, |v|\leq 1$ is applied after reducing $a\theta$
modulo $1$. Consequently,
$\mathbb D(u_\theta^a,1;Y)\leq(1/20)\sqrt{\mathcal L}$ for large $Y$.

Let $F= \overline{u_\theta} {\chi n^{it}}$. By the triangle inequality \eqref{GS_triangle_inequality}, we have the following power inequality
\[
 \mathbb D(u_\theta^a,\chi^a n^{iat};Y) =\mathbb D(1,F^a;Y) \le a\,\mathbb D(1,F;Y) = a\mathbb D(u_\theta,\chi n^{it};Y),
\]
and
\[
 \mathbb D(u_\theta^a,\chi^a n^{iat};Y) + \mathbb D(u_\theta^a,1;Y) =   \mathbb D(1,\overline{u}_\theta^a \chi^a n^{iat};Y) + \mathbb D(1, u_\theta^a;Y)  \geq\mathbb D(1,\chi^a n^{iat};Y).
\]

The character $\chi^a$ is again modulo $q$, and
$1\leq|at|\leq100Y$. It follows that
\begin{align*}
 a\mathbb D(u_\theta,\chi n^{it};Y)
 &\geq\mathbb D(u_\theta^a,\chi^a n^{iat};Y)\\
 &\geq\mathbb D(1,\chi^a n^{iat};Y)
          -\mathbb D(u_\theta^a,1;Y)\\
 &\geq\left(\frac25-\frac1{20}\right)\sqrt{\mathcal L}
  \geq\frac13\sqrt{\mathcal L}.
\end{align*}
As $a\leq100$, we conclude that
\[
 \mathbb D(u_\theta,\chi n^{it};Y)^2\geq\frac{\mathcal L}{90000}
 \qquad(1\leq|t|\leq Y).
\]
Since $\theta^2\leq1/4$, this is at least
$(4/90000)\theta^2\mathcal L$. Combining it with
\eqref{eq:small-height-distance} and taking that infimum proves \eqref{eq:full-distance}.
\end{proof}

For the rest of the paper put
\begin{equation}\label{eq:L-delta}
 L=\log\log(20N),\qquad\delta=L^{-2/5}.
\end{equation}
We always take $N$ large enough that $0<\delta\leq1/2$.
Lemma~\ref{lem:full-distance} implies
\begin{equation}\label{eq:moving-frequency-distance}
 \mathcal M(u_\theta;Y,Q)\gg L^{1/5}
\end{equation}
uniformly for $N\leq Y\leq20N$, $1\leq Q\leq(\log Y)^{1/125}$,
and $\delta\leq|\theta|\leq1/2$.

\section{Correlations averaged over short shifts}\label{sec:correlations}

\subsection{Integer correlations}

For fixed $N,k$ and $\bh\in\mathbb N^k$, write
\[
 P^{\rm int}_{\bh}(\theta)=\frac1N\sum_{n\leq N}
       \prod_{j=1}^ku_\theta(n+h_j),\qquad
 P^{\rm pr}_{\bh}(\theta)=\frac1{\pi(N)}\sum_{p\leq N}
       \prod_{j=1}^ku_\theta(p+h_j).
\]
Every summand has absolute value $1$. The triangle inequality
therefore gives $|P^{\rm int}_{\bh}(\theta)|\leq1$ and
$|P^{\rm pr}_{\bh}(\theta)|\leq1$ for every $\bh,\theta$.

\begin{proposition}\label{prop:integer-fourier}
For fixed $k\geq1$,
\begin{equation}\label{eq:integer-frequencies}
 \sup_{\log N\leq H\leq N}
 \sup_{\bh'\in[N]^{k-1}}
 \sup_{\delta\leq|\theta|\leq1/2}
 \E_{h_1\in[H]}|P^{\rm int}_{(h_1,\bh')}(\theta)|^2
 \ll_k\frac{\log L}{L}.
\end{equation}
\end{proposition}

\begin{proof}
First let $10\leq H\leq N$, fix $\theta\in\mathbb T$ and
$\bh'=(h_2,\ldots,h_k)\in[N]^{k-1}$, and put
\[
 B_{\theta,\bh'}(n)=\prod_{j=2}^ku_\theta(n+h_j).
\]
For $k=1$, this is the empty product and equals $1$.
In all cases $|B_{\theta,\bh'}(n)|=1$. This function need not
be multiplicative.

We apply \cite[Theorem 1.6, (1.11)]{MRT2015} on the Elliott's conjecture on average by taking
\[
 g_1=B_{\theta,\bh'},\qquad g_2=u_\theta,\qquad j_0=2,\quad X=N, \quad A=1.
\]
With
\[
 Q=\min\{(\log N)^{1/125},(\log H)^{20}\},\qquad
 M=\mathcal M(u_\theta;10N,Q),
\]
the cited estimate gives
\[
 \frac1H\sum_{h_1=1}^H
 \left|\frac1N\sum_{n=1}^N
       B_{\theta,\bh'}(n)u_\theta(n+h_1)\right|
 \ll e^{-M/80}+\frac{\log\log H}{\log H}
                    +(\log N)^{-1/3000}.
\]
That is,
\begin{equation}\label{eq:integer-first-fourier}
 \E_{h_1\in[H]}|P^{\rm int}_{(h_1,\bh')}(\theta)|
 \ll e^{-M/80}+\frac{\log\log H}{\log H}
                    +(\log N)^{-1/3000}.
\end{equation}
The implied constant is independent of the bounded function
$B_{\theta,\bh'}$. Thus the estimate is uniform in $\bh'$ and
$\theta$, including when these parameters depend on $N$.

Now assume $\log N\leq H\leq N$ and
$\delta\leq|\theta|\leq1/2$. By
\eqref{eq:moving-frequency-distance}, $M\gg L^{1/5}$.
Also,
\[
 \frac{\log\log H}{\log H}\ll\frac{\log L}{L},
 \qquad e^{-cL^{1/5}}+(\log N)^{-1/3000}
           \ll\frac{\log L}{L}.
\]
Finally, $|P^{\rm int}_{(h_1,\bh')}|\leq1$ implies
\[
 \E_{h_1\in[H]}|P^{\rm int}_{(h_1,\bh')}(\theta)|^2
 \leq\E_{h_1\in[H]}|P^{\rm int}_{(h_1,\bh')}(\theta)|
 \ll\frac{\log L}{L}.
\]
Taking the stated suprema proves \eqref{eq:integer-frequencies}.
No averaging over $h_2,\ldots,h_k$ is used.
\end{proof}

\subsection{Weights and the common fourth-moment estimate}

\begin{lemma}\label{lem:mass}
For all sufficiently large $N$, depending on $\ba$, one has
\begin{equation}\label{eq:mass-bound}
 0\leq C_{\ba}(N)\ll_{\ell,\ba}1,\qquad
 \frac1N\sum_{n\leq\sqrt N}w_{\ba}(n)
 \ll_{\ell,\ba}N^{-1/2}(\log N)^\ell.
\end{equation}
If $\ba$ is not admissible, then
\begin{equation}\label{eq:nonadmissible}
 C_{\ba}(N)\ll_{\ell,\ba}\frac{(\log N)^{\ell+1}}N.
\end{equation}
\end{lemma}

\begin{proof}
Take $N\geq2\max_i|a_i|+2$. Then $n+a_i\leq2N$ when $n\leq N$,
and, using the zero extension of $\Lambda$,
\begin{equation}\label{eq:weight-pointwise}
 0\leq w_{\ba}(n)\leq(\log(2N))^\ell.
\end{equation}
This immediately proves the second assertion of
\eqref{eq:mass-bound}.

If the tuple is not admissible, fix a prime $p_0$ for which
$r_{p_0}(\ba)=p_0$. For every $n$, at least one $n+a_i$ is
divisible by $p_0$. A nonzero weight then forces
$n+a_i=p_0^j$ for some $j\geq1$. There are only
$O_{\ell,\ba}(\log N)$ such choices of $n\leq N$.
Applying \eqref{eq:weight-pointwise} proves
\eqref{eq:nonadmissible}, and hence also the first assertion
of \eqref{eq:mass-bound} in this case.
We may therefore assume for the rest of the proof that $\ba$
is admissible.

The number of proper prime powers $p^j\leq2N$, $j\geq2$, is at most
\[
 \sum_{2\leq j\leq\log(2N)/\log2}(2N)^{1/j}
 \ll N^{1/2}\log N.
\]
Consequently the contribution to $C_{\ba}(N)$ from those $n$
for which at least one $n+a_i$ is a proper prime power is
$O_{\ell,\ba}(N^{-1/2}(\log N)^{\ell+1})$.
Every other nonzero summand has all $n+a_i$ prime.
The fixed-tuple upper-bound sieve
\cite[Theorem 7.16]{FI}, also recorded in
\cite[Lemma 2.3]{LT}, gives
\[
 \#\{n\leq N:n+a_i\text{ is prime for all }i\}
 \ll_{\ell,\ba}\frac{N}{(\log N)^\ell}.
\]
Multiplying by the upper bound
in \eqref{eq:weight-pointwise} proves $C_{\ba}(N)\ll_{\ell,\ba}1$.
\end{proof}

By the Sathe-Selberg theorem, we have the following estimate.  

\begin{lemma}\label{lem:exp-moment}
For every fixed $1<z<2$ and $M\geq2$,
\begin{equation}\label{eq:exp-moment}
 \sum_{m\leq M}z^{\Omega(m)}\ll_z M(\log(2M))^{z-1}.
\end{equation}
\end{lemma}

By \cite[Proposition 2.7]{LT}, we also have the following estimate on the  fourth moment of exponential sums.

\begin{proposition}\label{prop:fourth}
For $X\geq J\geq2$ and
any sequence $|b(n)|\leq1$,
\begin{align}
 \int_0^X\int_{\Torus}
 \abs{\sum_{x\leq n\leq x+J}b(n)w_{\ba}(n)e(\alpha n)}^4
 \,d\alpha\,dx \ll_{\ell,\ba}
 X\bigl(J^3+J^2(\log X)^{2\ell}+J(\log X)^{3\ell}\bigr).
 \label{eq:weighted-fourth}
\end{align}
In particular, the right-hand side is $O_{\ell,\ba}(XJ^3)$
if $J\geq(\log X)^{2\ell}$. The constants are independent of $b$.
\end{proposition}

For $\ell=1$ and $a_1=0$, Proposition~\ref{prop:fourth} gives,
for every sequence $|a(n)|\leq1$,
\begin{equation}
\begin{aligned}
 &\int_0^X\int_{\Torus}
 \abs{\sum_{x\leq n\leq x+J}a(n)\Lambda(n)e(\alpha n)}^4
 \,d\alpha\,dx\ll X\bigl(J^3+J^2(\log X)^2+J(\log X)^3\bigr).
\end{aligned}
 \label{eq:local-prime-fourth}
\end{equation}
This is $O(XJ^3)$ for $J\geq(\log X)^2$.
The specialization is used for the prime averages below; the
general version treats the fixed prime-tuple weights.

\subsection{ Short exponential sums}

We cite a Fourier-analytic result in \cite[Proposition 3.2]{LT}. 

\begin{lemma}[{\cite[Proposition 3.2]{LT}}]\label{lem:transfer}
Let $h\geq2$ be an integer and $X\geq2h$. Let $F,G:\Z\to\mathbb C$
be supported on $[1,X]$ and $[2h,X]$, respectively, with $|F|\leq1$.
Suppose that
\begin{align}
 \sup_{\alpha\in\Torus}\int_0^X
 \abs{\sum_{x\leq n\leq x+2h}F(n)e(\alpha n)}^2dx
 &\leq\eta h^2X,\label{eq:local-H2}\\
 \int_0^X\int_{\Torus}
 \abs{\sum_{x\leq n\leq x+2h}G(n)e(\alpha n)}^4d\alpha\,dx
 &\leq C_1h^3X.\label{eq:local-H1}
\end{align}
Then
\begin{equation}\label{eq:local-transfer}
 \sum_{r=1}^h\abs{\sum_m F(m+r)G(m)}
 \leq2^{3/4}C_1^{1/4}\eta^{1/4}hX.
\end{equation}
\end{lemma}

\begin{lemma}\label{lem:local-fourier}
Let
\[
 h_0=\left\lfloor\tfrac14\exp(\sqrt{\log N})\right\rfloor,
 \qquad X_0=4N.
\]
For $\delta\leq|\theta|\leq1/2$ and every integer $0\leq b\leq N$,
let $F_{\theta,b}(n)=u_\theta(n+b)$ for $1\leq n\leq4N$ and
$F_{\theta,b}(n)=0$ otherwise. There is $\eta=\eta(N)>0$,
independent of $\theta,b$, such that
\[
 \sup_{\alpha\in\Torus}\int_0^{X_0}
 \abs{\sum_{x\leq n\leq x+2h_0}F_{\theta,b}(n)e(\alpha n)}^2dx
 \leq\eta h_0^2X_0,
 \qquad \eta\ll_B L^{-B}
\]
for every fixed $B>0$.
\end{lemma}

\begin{proof}
By the short exponential sum estimate established in \cite[Theorem 1.7]{MRT2015}, for a $1$-bounded
multiplicative $g$ and $Y\geq J\geq10$, we have
\begin{align*}
 \sup_{\alpha\in\mathbb T}\int_0^Y
  \left|\sum_{x\leq n\leq x+J}g(n)e(\alpha n)\right|dx \ll JY\left(
  e^{-\mathcal M(g;Y,Q_J)/20}
  +\frac{\log\log J}{\log J}+(\log Y)^{-1/700}\right),
\end{align*}
where
\[
Q_J=\min\{(\log Y)^{1/125},(\log J)^5\}.
\]
Take $g=u_\theta$, $Y=6N$,  $J=2h_0$, and
\(
 Q_0=\min\{(\log(6N))^{1/125},(\log(2h_0))^5\}.
\)
By \eqref{eq:moving-frequency-distance}, which gives
$\mathcal M(u_\theta;6N,Q_0)\gg L^{1/5}$, it follows that
\begin{equation}\label{eq:MRT-short}
 \sup_{\alpha\in\mathbb T}\int_0^{6N}
 \left|\sum_{x\leq m\leq x+2h_0}u_\theta(m)e(\alpha m)\right|dx
 \ll h_0N\eta_0,
\end{equation}
where we may choose absolute constants $C,c>0$ and set
\[
 \eta_0=C\left(e^{-cL^{1/5}}
       +\frac{\log\log(2h_0)}{\log(2h_0)}
       +(\log(6N))^{-1/700}\right).
\]

For an integer $0\leq b\leq N$, define on $\mathbb Z$
\[
 F_{\theta,b}(n)=
 \begin{cases}u_\theta(n+b),&1\leq n\leq4N,\\0,&\text{otherwise}.
 \end{cases}
\]
Then $|F_{\theta,b}|\leq1$ and its support is contained in $[1,X_0]$.
Fix $\alpha$, and let
\[
 U_b(x,\alpha)=\sum_{x\leq n\leq x+2h_0}
                      F_{\theta,b}(n)e(\alpha n).
\]
For almost every $x\in[0,4N-2h_0]$, substituting $m=n+b$ gives
\[
 U_b(x,\alpha)=e(-\alpha b)
       \sum_{x+b\leq m\leq x+b+2h_0}u_\theta(m)e(\alpha m).
\]
Since $0\leq b\leq N$, $x+b\in [0,5N]$.
It follows from \eqref{eq:MRT-short} that
\[
 \int_0^{4N-2h_0}|U_b(x,\alpha)|\,dx\ll h_0N\eta_0.
\]
By the trivial bound $|U_b(x,\alpha)|\leq2h_0+1$,
\begin{align*}
 \int_0^{4N-2h_0}|U_b(x,\alpha)|^2dx
 &\leq(2h_0+1)\int_0^{4N-2h_0}|U_b(x,\alpha)|dx \ll h_0^2N\eta_0.
\end{align*}
The trivial bound also gives
\[
 \int_{4N-2h_0}^{4N}|U_b(x,\alpha)|^2dx
 \leq2h_0(2h_0+1)^2\ll h_0^3.
\]
Thus, 
\begin{equation}\label{eq:truncated-local-fourier}
 \sup_{\alpha\in\mathbb T}\int_0^{4N}|U_b(x,\alpha)|^2dx
 \ll h_0^2N\left(\eta_0+\frac{h_0}{N}\right).
\end{equation}
The constants do not depend on $b$. 
This verifies \eqref{eq:local-H2} with
$F=F_{\theta,b}$, $X=X_0=4N$, $h=h_0$, and  $\eta=C'(\eta_0+h_0/N)$ for some large enough constant $C'>0$. Moreover, by the choice of $\eta_0$ and $h_0$, for every fixed $B>0$ we have
\begin{equation}\label{eq:eta-decay}
 \eta\ll_B L^{-B}.
\end{equation}

\end{proof}

\subsection{Prime and von Mangoldt weighted correlations}

\begin{proposition}\label{prop:prime-fourier}
For every fixed $A>0$ and $k\geq1$,
\begin{equation}\label{eq:prime-frequencies}
 \sup_{\exp(\sqrt{\log N})\leq H\leq N}
 \sup_{\bh'\in[N]^{k-1}}
 \sup_{\delta\leq|\theta|\leq1/2}
 \E_{h_1\in[H]}|P^{\rm pr}_{(h_1,\bh')}(\theta)|^2
 \ll_{A,k}L^{-A}.
\end{equation}
\end{proposition}

\begin{proof}
Put
\[
 h_0=\left\lfloor\frac14\exp(\sqrt{\log N})\right\rfloor,
 \qquad X_0=4N.
\]
Then for sufficiently large $N$,
\[
 \log(2h_0)=\sqrt{\log N}+O(1),\qquad
 (\log X_0)^2\leq2h_0<\sqrt N.
\]
And every allowed $H$ satisfies $h_0\leq H/4$.

Use Lemma~\ref{lem:local-fourier} with its $F_{\theta,b}$ and $\eta$.
In particular, \eqref{eq:local-H2} holds uniformly in $b$ and $\theta$,
and $\eta$ satisfies \eqref{eq:eta-decay} for every fixed $B>0$.

Now, we fix $h_2,\ldots,h_k\in[N]$ and define, with zero extension,
\[
 G_\theta(n)=(\log N)
   \mathbf1_{\{\sqrt N<n\leq N,\ n\ {\rm prime}\}}
                 \prod_{j=2}^ku_\theta(n+h_j).
\]
The product is interpreted as $1$ when $k=1$. Then  $\operatorname{supp}G_\theta \subset [2h_0,X_0]$, since  $2h_0<\sqrt N$.
On this support, $n=p>\sqrt N$ is prime, so
$\log N\leq2\log p=2\Lambda(p)$, and hence $|G_\theta(n)|\leq2\Lambda(n)$ for all $n\ge1$.
Put
\[
 a_\theta(n)=
 \begin{cases}G_\theta(n)/(2\Lambda(n)),&\Lambda(n)>0,\\
               0,&\Lambda(n)=0.
 \end{cases}
\]
Then $|a_\theta(n)|\leq1$ and $G_\theta(n)=2a_\theta(n)\Lambda(n)$.
This holds for all $h_2,\ldots,h_k$.

Applying \eqref{eq:local-prime-fourth} with $X=X_0$, $J=2h_0\geq(\log X_0)^2$,
and $a=a_\theta$, we get 
\begin{align*}
 \quad\int_0^{4N}\int_{\mathbb T}
  \left|\sum_{x\leq n\leq x+2h_0}G_\theta(n)e(\alpha n)\right|^4
                  d\alpha\,dx \ll X_0h_0^3.
\end{align*}
This verifies \eqref{eq:local-H1} with an absolute $C_1$ independent
of $b,\theta,h_2,\ldots,h_k$.

All hypotheses of Lemma~\ref{lem:transfer} have now been verified.
With $F=F_{\theta,b}$, $G=G_\theta$, $X=X_0$, and $h=h_0$, by  \eqref{eq:local-transfer} we obtain
\[
 \sum_{r=1}^{h_0}\left|\sum_n F_{\theta,b}(n+r)G_\theta(n)\right|
 \ll\eta^{1/4}h_0N.
\]
If $G_\theta(n)\neq0$ and $r\leq h_0$, then   $\sqrt{N} < n \leq N$ and
$1\leq n+r\leq N+h_0<4N$. The inner sum above is exactly
\[
 (\log N)\sum_{\sqrt N<p\leq N}
          u_\theta(p+b+r)\prod_{j=2}^ku_\theta(p+h_j).
\]
Dividing by $h_0\pi(N)\log N$ and using
$N/(\pi(N)\log N)\ll1$, we obtain
\[
 \frac1{h_0}\sum_{r=1}^{h_0}
 \left|\frac1{\pi(N)}\sum_{\sqrt N<p\leq N}
        u_\theta(p+b+r)\prod_{j=2}^ku_\theta(p+h_j)\right|
 \ll\eta^{1/4}.
\]
By the prime number theorem, we have
$\pi(\sqrt N)/\pi(N)\ll  N^{-1/2}$.
Therefore,
\begin{equation}\label{eq:prime-block}
 \frac1{h_0}\sum_{r=1}^{h_0}
 \left|\frac1{\pi(N)}\sum_{p\leq N}
          u_\theta(p+b+r)\prod_{j=2}^ku_\theta(p+h_j)\right|
 \ll\eta^{1/4}+N^{-1/2}.
\end{equation}
This estimate is uniform for every integer $0\leq b\leq N$
and  $h_2,\dots,h_k\in [N]$. By the definition of $P^{\rm pr}_{\bh}(\theta)$, \eqref{eq:prime-block}  is equivalent to saying that
\[
 \frac1{h_0}\sum_{r=1}^{h_0}
 \left| P^{\rm pr}_{(b+r, h_2,\dots,h_k)}(\theta)  \right|
 \ll\eta^{1/4}+N^{-1/2}.
\]

Let $\exp(\sqrt{\log N})\leq H\leq N$ and
$M_H=\lceil H/h_0\rceil$.
The blocks
\[
 \{ih_0+1,\ldots,(i+1)h_0\},\qquad 0\leq i<M_H,
\]
cover $[H]$. Their starting points satisfy
$0\leq ih_0<H\leq N$, so \eqref{eq:prime-block} applies to
every block. Also,
$M_Hh_0\leq H+h_0\leq2H$.
For each fixed $\bh'\in[N]^{k-1}$,
\begin{align*}
 \frac1H\sum_{h_1=1}^H|P^{\rm pr}_{\bh}(\theta)|
 &\leq\frac1H\sum_{i=0}^{M_H-1}\sum_{r=1}^{h_0}
       |P^{\rm pr}_{(ih_0+r,h_2,\ldots,h_k)}(\theta)|\\
 &\ll\frac{M_Hh_0}{H}(\eta^{1/4}+N^{-1/2})
 \ll\eta^{1/4}+N^{-1/2}.
\end{align*}
The remaining shifts stay fixed throughout this argument.
By $|P^{\rm pr}_{\bh}|^2\leq|P^{\rm pr}_{\bh}|$, it follows that
\[
 \E_{h_1\in[H]}|P^{\rm pr}_{(h_1,\bh')}(\theta)|^2
 \ll\eta^{1/4}+N^{-1/2}.
\]
Then \eqref{eq:prime-frequencies} follows by  \eqref{eq:eta-decay} with $B=4A$ and 
$N^{-1/2}\ll_A L^{-A}$. The constants are independent of
$\bh'\in[N]^{k-1}$, so taking the supremum over these shifts is valid.
\end{proof}

For $\bh\in\N^k$, define
\begin{equation}\label{eq:P-definition}
 P^{\rm wt}_{\bh}(\theta)=\frac1N\sum_{n=1}^{N}
    w_{\ba}(n)\prod_{j=1}^{k}u_\theta(n+h_j).
\end{equation}
Lemma~\ref{lem:mass} and the triangle inequality give
\begin{equation}\label{eq:P-trivial}
 |P^{\rm wt}_{\bh}(\theta)|\leq C_{\ba}(N)\ll_{\ell,\ba}1.
\end{equation}

Using the method in the proof of Proposition \ref{prop:prime-fourier}, we prove the same result for $P^{\rm wt}_{\bh}$.

\begin{proposition}\label{prop:weighted-fourier}
For every fixed $A>0$,
\begin{equation}\label{eq:weighted-fourier}
 \sup_{\exp(\sqrt{\log N})\leq H\leq N}
 \sup_{\bh'\in[N]^{k-1}}
 \sup_{\delta\leq|\theta|\leq1/2}
 \E_{h_1\in[H]}|P^{\rm wt}_{(h_1,\bh')}(\theta)|^2
 \ll_{A,k,\ell,\ba}L^{-A}.
\end{equation}
\end{proposition}

\begin{proof}
Fix $\delta\leq|\theta|\leq1/2$, and put
\[
 h_0=\left\lfloor\tfrac14\exp(\sqrt{\log N})\right\rfloor,
 \qquad X_0=4N.
\]
For sufficiently large $N$ in terms of $\ell,\ba$,
\begin{equation}\label{eq:block-size}
 \log(2h_0)=\sqrt{\log N}+O(1),\qquad
 10\leq(\log X_0)^{2\ell}\leq2h_0<\sqrt N,
 \qquad h_0\leq H/4.
\end{equation}
Indeed, $\sqrt{\log N}$ eventually exceeds
$2\ell\log\log(4N)$ and is smaller than $(\log N)/2$.

Lemma~\ref{lem:local-fourier} supplies $F_{\theta,b}$ and
$\eta\ll_B L^{-B}$ uniformly for $0\leq b\leq N$ and
$\delta\leq|\theta|\leq1/2$. Thus \eqref{eq:local-H2} holds
with $X=X_0$ and $h=h_0$.

Fix $h_2,\ldots,h_k\in[N]$, and define
\[
 G(n)=\one_{\sqrt N<n\leq N}\,w_{\ba}(n)
                   \prod_{j=2}^k u_\theta(n+h_j).
\]
For $k=1$ the product is empty and equals $1$.
By \eqref{eq:block-size}, $G$ is supported in $[2h_0,X_0]$.
The factors other than $w_{\ba}$ form a coefficient sequence
of absolute value at most one. Proposition~\ref{prop:fourth},
with $X=X_0$ and $J=2h_0$, therefore gives
\[
 \int_0^{X_0}\int_{\Torus}
 \abs{\sum_{x\leq n\leq x+2h_0}G(n)e(\alpha n)}^4d\alpha\,dx
 \ll_{\ell,\ba}X_0h_0^3.
\]
This is \eqref{eq:local-H1} with a constant $C_1$ independent
of $b,\theta,h_2,\ldots,h_k$.

Lemma~\ref{lem:transfer} now implies
\[
 \sum_{r=1}^{h_0}\abs{\sum_nF_{\theta,b}(n+r)G(n)}
 \ll_{\ell,\ba}\eta^{1/4}h_0N.
\]
On the support of $G$, $1\leq n+r\leq N+h_0<4N$, so the inner
sum equals
\[
 \sum_{\sqrt N<n\leq N}w_{\ba}(n)
       u_\theta(n+b+r)\prod_{j=2}^ku_\theta(n+h_j).
\]
Divide by $Nh_0$. Restoring the omitted $n\leq\sqrt N$
costs at most
$N^{-1/2}(\log(2N))^\ell\ll_\ell N^{-1/3}$ by
Lemma~\ref{lem:mass}. It follows that
\begin{equation}\label{eq:block-correlation}
 \frac1{h_0}\sum_{r=1}^{h_0}
 \abs{P^{\rm wt}_{(b+r,h_2,\ldots,h_k)}(\theta)}
 \ll_{\ell,\ba}\eta^{1/4}+N^{-1/3},
\end{equation}
uniformly for every integer $0\leq b\leq N$.

Let $M_H=\lceil H/h_0\rceil$. The blocks
$\{ih_0+1,\ldots,(i+1)h_0\}$, $0\leq i<M_H$, cover $[H]$.
Their starting points satisfy $0\leq ih_0<H\leq N$, so
\eqref{eq:block-correlation} applies to each block.
Moreover $M_Hh_0\leq H+h_0\leq2H$.
 Therefore, for each fixed $\bh'\in[N]^{k-1}$,
\[
 \frac1H\sum_{h_1=1}^H|P^{\rm wt}_{\bh}(\theta)|
 \ll_{\ell,\ba}\frac{M_Hh_0}{H}
                       (\eta^{1/4}+N^{-1/3})
 \ll_{\ell,\ba}\eta^{1/4}+N^{-1/3}.
\]
Keep the remaining shifts fixed. Since
$|P^{\rm wt}_{\bh}|^2\leq C_{\ba}(N)|P^{\rm wt}_{\bh}|$,
Lemma~\ref{lem:mass} gives
\[
 \E_{h_1\in[H]}|P^{\rm wt}_{(h_1,\bh')}(\theta)|^2
 \ll_{\ell,\ba}\eta^{1/4}+N^{-1/3}.
\]
Use \eqref{eq:eta-decay} with $B=4A$ and
$N^{-1/3}\ll_A L^{-A}$. All estimates are independent of
$H$, $\theta$, and $\bh'\in[N]^{k-1}$ in the permitted ranges,
proving the proposition. The arbitrary bounded coefficients in
Proposition~\ref{prop:fourth} are what allow the other shifts to
remain fixed, even when some of them coincide.
\end{proof}

\section{Averaged shift invariance}\label{sec:invariance}

 For $\bh\in[N]^k$,
define
\begin{equation}\label{eq:rho-weighted}
 \rho^{\rm int}_{\bh}(r)=\frac1N
       \sum_{n\leq N}\one_{\{S_{\bh}(n)=r\}}, \, \rho^{\rm pr}_{\bh}(r)=\frac1{\pi(N)}
       \sum_{p\leq N}\one_{\{S_{\bh}(p)=r\}}, \,  \rho^{\rm wt}_{\bh}(r)=\frac1N
       \sum_{n\leq N}w_{\ba}(n)\one_{\{S_{\bh}(n)=r\}}.
\end{equation}
They are nonnegative and finitely supported on $\Z_{\geq0}$;
extend each by zero to negative integers. Their common masses
within each family are
\begin{equation}\label{eq:common-masses}
 \sum_r\rho^\star_{\bh}(r)=m^\star_N,\, \star\in\{{\rm int},{\rm pr},{\rm wt}\};\quad
 m^{\rm int}_N=m^{\rm pr}_N=1,\qquad
 m^{\rm wt}_N=C_{\ba}(N)\ll_{\ell,\ba}1.
\end{equation}
The final bound follows from Lemma~\ref{lem:mass}.
In particular $|P^\star_{\bh}(\theta)|\leq m^\star_N$.

For $\star\in\{{\rm int},{\rm pr},{\rm wt}\}$, we have
\[
 A^\star_{N,\bh}f(x)=\sum_{r\geq0}\rho^\star_{\bh}(r)f(T^rx).
\]
To prove Theorems~\ref{thm:integers}, \ref{thm:primes}  and \ref{thm:weighted}, we mainly estimate
$$\sum_{r\in\Z}|\rho^\star_{\bh}(r)-\rho^\star_{\bh}(r-1)|.$$
Put
\[
 \Delta\rho^\star_{\bh}(r)=\rho^\star_{\bh}(r)
                          -\rho^\star_{\bh}(r-1).
\]
For any $s\geq1$, write
$\|\Delta\rho^\star_{\bh}\|_{\ell^s}
=(\sum_{r\in\Z}|\Delta\rho^\star_{\bh}(r)|^s)^{1/s}$.
For every fixed $\bh'\in[N]^{k-1}$, define
\[
 D^\star_{N,H}(\bh')=\E_{h_1\in[H]}
                   \|\Delta\rho^\star_{(h_1,\bh')}\|_{\ell^1},
 \qquad \star\in\{{\rm int},{\rm pr},{\rm wt}\}.
\]
All three families now use the same single-shift average.

 First, we can write  $P^\star_{\bh}$ as a Fourier expansion of $\rho^\star_{\bh}$.  For the integer distribution,
\begin{align*}
 \sum_r\rho^{\rm int}_{\bh}(r)e(r\theta)
 =\frac1N\sum_{n\leq N}\sum_r
                 \mathbf1_{\{S_{\bh}(n)=r\}}e(r\theta)=\frac1N\sum_{n\leq N}e(S_{\bh}(n)\theta)
  =P^{\rm int}_{\bh}(\theta).
\end{align*}
The same calculation, with prime summation or with the weight
$w_{\ba}(n)$ inserted, gives
$\sum_r\rho^\star_{\bh}(r)e(r\theta)=P^\star_{\bh}(\theta)$
also for $\star={\rm pr},{\rm wt}$. Thus for all three measures, replacing $r-1$ by $r$ shows that
\[
 \sum_r\rho^\star_{\bh}(r-1)e(r\theta)
 =\sum_r\rho^\star_{\bh}(r)e((r+1)\theta)
 =e(\theta)P^\star_{\bh}(\theta).
\]
Subtracting yields
\[
 \sum_r\Delta\rho^\star_{\bh}(r)e(r\theta)
 =(1-e(\theta))P^\star_{\bh}(\theta).
\]

By Parseval's formula,
\[
\|\Delta\rho^\star_{\bh}\|_{\ell^2}^2
 =\int_{-1/2}^{1/2}|1-e(\theta)|^2
          |P^\star_{\bh}(\theta)|^2\,d\theta.
\]
Hold $\bh'\in[N]^{k-1}$ fixed and average only over $h_1\in[H]$.
Interchanging this finite average with the integral gives
\begin{equation}\label{eq:difference-parseval}
 \E_{h_1\in[H]}\|\Delta\rho^\star_{(h_1,\bh')}\|_{\ell^2}^2
 =\int_{-1/2}^{1/2}|1-e(\theta)|^2
       \E_{h_1\in[H]}|P^\star_{(h_1,\bh')}(\theta)|^2\,d\theta.
\end{equation}
This identity holds for each of the three families.
For $|\theta|\leq\delta$, we use
\[
 |1-e(\theta)|=2|\sin(\pi\theta)|\leq2\pi|\theta|,
 \qquad |P^\star_{(h_1,\bh')}(\theta)|\leq m_N^\star.
\]
The contribution from this range is therefore at most
\[
 4\pi^2(m_N^\star)^2\int_{-\delta}^{\delta}\theta^2\,d\theta
 =\frac{8\pi^2}{3}(m_N^\star)^2\delta^3.
\]
On $\delta<|\theta|\leq1/2$, use $|1-e(\theta)|\leq2$.
Since this part of the torus has measure at most $1$, we obtain
\begin{equation}\label{eq:parseval-split}
\begin{aligned}
 \E_{h_1\in[H]}\|\Delta\rho^\star_{(h_1,\bh')}\|_{\ell^2}^2 \leq\frac{8\pi^2}{3}(m_N^\star)^2\delta^3+4\sup_{\delta\leq|\theta|\leq1/2}
       \E_{h_1\in[H]}|P^\star_{(h_1,\bh')}(\theta)|^2.
\end{aligned}
\end{equation}
Recall that $L=\log\log(20N)$ and $\delta=L^{-2/5}$.
In particular, the low-frequency contribution in \ref{eq:parseval-split} is uniformly
$O(L^{-6/5})$, since the masses are bounded by
\eqref{eq:common-masses}.

Now, we will use  \ref{eq:parseval-split}  to estimate  $D^\star_{N,H}(\bh')=\E_{h_1\in[H]}
                   \|\Delta\rho^\star_{(h_1,\bh')}\|_{\ell^1}$ for $\star ={\rm int},{\rm pr},{\rm wt}$, respectively.

\subsection{Integer measure}

\begin{proposition}\label{prop:integer-invariance}
For any $k\geq1$,
\begin{equation}\label{eq:integer-invariance}
 \sup_{\log N\leq H\leq N}\sup_{\bh'\in[N]^{k-1}}D^{\rm int}_{N,H}(\bh')
 \ll_k L^{-1/5}\sqrt{\log L}.
\end{equation}
\end{proposition}

\begin{proof}
Fix $\log N\leq H\leq N$ and $\bh'\in[N]^{k-1}$, and write
$\bh=(h_1,\bh')$. By Proposition~\ref{prop:integer-fourier} and
\eqref{eq:parseval-split}, we have
\begin{equation}\label{eq:integer-difference-l2}
 \E_{h_1\in[H]}\|\Delta\rho^{\rm int}_{\bh}\|_{\ell^2}^2
 \ll_k\delta^3+\frac{\log L}{L}
 =L^{-6/5}+\frac{\log L}{L}
 \ll_k\frac{\log L}{L}.
\end{equation}

For any $h_j\leq N$, the interval $[h_j+1,h_j+N]$ is contained
in $[1,2N]$, by the Tur\'an inequality (e.g.,  \cite[Theorem 3.1.2]{CojocaruMurty2006}) we have
\[
 \frac1N\sum_{n\leq N}(\Omega(n+h_j)-L)^2
 \leq\frac1N\sum_{m\leq2N}(\Omega(m)-L)^2\ll L.
\]
For each fixed $\bh$, by the definition of $\rho^{\rm int}_{\bh}$
and the Cauchy--Schwarz inequality,
\begin{equation}\label{eq:centered-moment}
\begin{aligned}
 \sum_r(r-kL)^2\rho^{\rm int}_{\bh}(r)
 &=\frac1N\sum_{n\leq N}(S_{\bh}(n)-kL)^2\\
 &=\frac1N\sum_{n\leq N}
        \left(\sum_{j=1}^k(\Omega(n+h_j)-L)\right)^2\\
 &\leq\frac{k}{N}\sum_{j=1}^k\sum_{n\leq N}
                     (\Omega(n+h_j)-L)^2\ll_k L.
\end{aligned}
\end{equation}
This bound holds for every $\bh\in[N]^k$, without averaging
over any coordinate. 

Let $R\geq2$ and
$\mathcal I_R=\{r\in\mathbb Z:|r-kL|\leq R\}$.
This interval contains at most $2R+1$ integers.
For each $\bh$, Cauchy--Schwarz in $r$ gives
\[
 \sum_{r\in\mathcal I_R}|\Delta\rho^{\rm int}_{\bh}(r)|
 \leq\sqrt{2R+1}
       \left(\sum_{r\in\mathcal I_R}
                  |\Delta\rho^{\rm int}_{\bh}(r)|^2\right)^{1/2}
 \leq\sqrt{2R+1}\|\Delta\rho^{\rm int}_{\bh}\|_{\ell^2}.
\]
Cauchy--Schwarz  also gives
\[
 \left(\E_{h_1\in[H]}\|\Delta\rho^{\rm int}_{\bh}\|_{\ell^2}\right)^2
 \leq\E_{h_1\in[H]}\|\Delta\rho^{\rm int}_{\bh}\|_{\ell^2}^2.
\]
Combining these inequalities with
\eqref{eq:integer-difference-l2}, we obtain
\[
 \E_{h_1\in[H]}\sum_{r\in\mathcal I_R}|\Delta\rho^{\rm int}_{\bh}(r)|
 \ll_k\sqrt{\frac{R\log L}{L}}.
\]

For $r\notin \mathcal I_R$,  for each $\bh$,
  by \eqref{eq:centered-moment} we have
\begin{align*}
 \sum_{|r-kL|>R}|\Delta\rho^{\rm int}_{\bh}(r)|
 &\leq\sum_{|r-kL|>R}\rho^{\rm int}_{\bh}(r)
       +\sum_{|r-kL|>R}\rho^{\rm int}_{\bh}(r-1)\\
 &\leq2\sum_{|s-kL|>R-1}\rho^{\rm int}_{\bh}(s)\\
 &\leq\frac2{(R-1)^2}\sum_s(s-kL)^2\rho^{\rm int}_{\bh}(s)
 \ll_k\frac{L}{R^2}.
\end{align*}
Averaging this estimate only over $h_1$ and adding the estimate
over $r\in\mathcal I_R$ yields
\[
 D^{\rm int}_{N,H}(\bh')
 \ll_k\sqrt{\frac{R\log L}{L}}+\frac{L}{R^2}.
\]
Take $R=L^{3/5}$. The two contributions are respectively
$O_k(L^{-1/5}\sqrt{\log L})$ and $O_k(L^{-1/5})$.
Their constants are independent of $H$ and $\bh'$, so taking
the required suprema proves \eqref{eq:integer-invariance}. 
\end{proof}

\subsection{Exponential tails for prime and weighted measure}

The tail estimate used below holds for each individual vector of shifts.
Thus no averaging over the fixed coordinates $h_2,\ldots,h_k$ is needed.

\begin{lemma}\label{lem:rho-tail}
Let $\star\in\{{\rm pr},{\rm wt}\}$, and put
$d_{\rm pr}=1$ and $d_{\rm wt}=\ell$.
Uniformly for $\bh\in[N]^k$, $R\geq0$, and every fixed $1<z<2$,
\begin{equation}\label{eq:rho-tail}
 \sum_{r>R}\rho^\star_{\bh}(r)
 \ll_{k,\ell,\ba,z}(\log N)^{d_\star+z-1}z^{-R/k}.
\end{equation}
In the prime case the implied constant depends only on $k$ and $z$.
In particular, with
\[
 D_\star=\frac{k(d_\star+4)}{\log(3/2)},
 \qquad R_N^\star=\lceil D_\star L\rceil,
\]
we have, uniformly for $\bh\in[N]^k$,
\begin{equation}\label{eq:rho-tail-log}
 \sum_{r>R_N^\star-1}\rho^\star_{\bh}(r)
 \ll_{k,\ell,\ba}(\log N)^{-2}.
\end{equation}
Again, in the prime case the implied constant depends only on $k$.
\end{lemma}

\begin{proof}
For $1\leq n\leq N$, write
\[
 v_N^{\rm pr}(n)=\frac{N}{\pi(N)}
                   \one_{\{n\text{ is prime}\}},
 \qquad v_N^{\rm wt}(n)=w_{\ba}(n).
\]
Then both distributions have the common representation
\[
 \rho^\star_{\bh}(r)
 =\frac1N\sum_{n\leq N}v_N^\star(n)
                         \one_{\{S_{\bh}(n)=r\}}.
\]
The bound $\pi(N)\gg N/\log N$ and
\eqref{eq:weight-pointwise} give
\[
 0\leq v_N^\star(n)\ll_{\ell,\ba}(\log N)^{d_\star},
\]
with an absolute implied constant in the prime case.
If $S_{\bh}(n)>R$, then $\Omega(n+h_j)>R/k$ for at least one
$j\in[k]$. Since $z>1$, the union bound gives
\begin{align*}
 \sum_{r>R}\rho^\star_{\bh}(r)
 &\leq\frac1N\sum_{j=1}^k\sum_{n\leq N}v_N^\star(n)
                  \one_{\{\Omega(n+h_j)>R/k\}}\\
 &\ll_{\ell,\ba}\frac{(\log N)^{d_\star}}{N}z^{-R/k}
                  \sum_{j=1}^k\sum_{n\leq N}z^{\Omega(n+h_j)}\\
 &\ll_{k,\ell,\ba}\frac{(\log N)^{d_\star}}{N}z^{-R/k}
                  \sum_{m\leq2N}z^{\Omega(m)}.
\end{align*}
By Lemma~\ref{lem:exp-moment},
\[
 \sum_{m\leq2N}z^{\Omega(m)}
 \ll_z N(\log N)^{z-1}.
\]
This proves \eqref{eq:rho-tail}.

Now set $z=3/2$. Since $R_N^\star-1\geq D_\star L-1$,
\[
 z^{-(R_N^\star-1)/k}
 \leq z^{1/k}\exp\left(-\frac{D_\star\log z}{k}L\right)
 =z^{1/k}(\log(20N))^{-(d_\star+4)}.
\]
Substituting $R=R_N^\star-1$ in \eqref{eq:rho-tail} therefore gives
\[
 \sum_{r>R_N^\star-1}\rho^\star_{\bh}(r)
 \ll_{k,\ell,\ba}
 (\log N)^{d_\star+1/2}
 (\log(20N))^{-(d_\star+4)}
 \ll_{k,\ell,\ba}(\log N)^{-2}.
\]
This proves \eqref{eq:rho-tail-log}.
\end{proof}

\subsection{Prime measure}

\begin{proposition}\label{prop:prime-invariance}
For fixed $k\geq1$,
\begin{equation}\label{eq:prime-invariance}
 \sup_{\exp(\sqrt{\log N})\leq H\leq N}
 \sup_{\bh'\in[N]^{k-1}}D^{\rm pr}_{N,H}(\bh')
 \ll_k L^{-1/10}.
\end{equation}
\end{proposition}

\begin{proof}
Fix $\exp(\sqrt{\log N})\leq H\leq N$ and
$\bh'=(h_2,\ldots,h_k)\in[N]^{k-1}$.
Apply Proposition~\ref{prop:prime-fourier} with $A=4$ and
\eqref{eq:parseval-split}. Since $m_N^{\rm pr}=1$, we obtain
\begin{equation}\label{prime_invariance_ell_2}
 \E_{h_1\in[H]}
       \|\Delta\rho^{\rm pr}_{(h_1,\bh')}\|_{\ell^2}^2
 \ll_k\delta^3+L^{-4}
 \ll_k L^{-6/5}.
\end{equation}
The implied constant is independent of both $H$ and $\bh'$.

Take $R=R_N^{\rm pr}$ from Lemma~\ref{lem:rho-tail}, then
$R\asymp_k L$.
For the central range $0\leq r\leq R$, Cauchy--Schwarz first in
$r$ gives
\begin{align*}
 &\E_{h_1\in[H]}\sum_{r=0}^{R}
                   |\Delta\rho^{\rm pr}_{(h_1,\bh')}(r)|\\
 &\quad\leq\sqrt{R+1}\,
       \E_{h_1\in[H]}
                   \|\Delta\rho^{\rm pr}_{(h_1,\bh')}\|_{\ell^2}\\
 &\quad\leq\sqrt{R+1}
       \left(\E_{h_1\in[H]}
                   \|\Delta\rho^{\rm pr}_{(h_1,\bh')}\|_{\ell^2}^2
       \right)^{1/2}
 \ll_k L^{1/2}L^{-3/5}=L^{-1/10}.
\end{align*}
For the tail, 
for each $h_1\in[H]$,
\begin{align*}
 \sum_{r>R}|\Delta\rho^{\rm pr}_{(h_1,\bh')}(r)|
 &\leq\sum_{r>R}\rho^{\rm pr}_{(h_1,\bh')}(r)
       +\sum_{r>R}\rho^{\rm pr}_{(h_1,\bh')}(r-1)\\
 &\leq2\sum_{s>R-1}\rho^{\rm pr}_{(h_1,\bh')}(s)
 \ll_k(\log N)^{-2},
\end{align*}
where the last bound follows from \eqref{eq:rho-tail-log}.
Averaging this inequality only over $h_1$ and adding the previous
estimate proves
\[
 D^{\rm pr}_{N,H}(\bh')
 \ll_k L^{-1/10}+(\log N)^{-2}
 \ll_k L^{-1/10}.
\]
All the estimates are uniform in $H$ and $\bh'$, proving
\eqref{eq:prime-invariance}.
\end{proof}

\subsection{Von Mangoldt weighted measure}

\begin{proposition}\label{prop:weighted-invariance}
For fixed $k,\ell\geq1$ and fixed distinct integers
$a_1,\ldots,a_\ell$,
\begin{equation}\label{eq:weighted-invariance}
 \sup_{\exp(\sqrt{\log N})\leq H\leq N}
 \sup_{\bh'\in[N]^{k-1}}D^{\rm wt}_{N,H}(\bh')
 \ll_{k,\ell,\ba}L^{-1/10}.
\end{equation}
\end{proposition}

\begin{proof}
Fix $\exp(\sqrt{\log N})\leq H\leq N$ and
$\bh'\in[N]^{k-1}$. Apply \eqref{eq:parseval-split} with
$\star={\rm wt}$. Lemma~\ref{lem:mass} bounds the mass
$m_N^{\rm wt}=C_{\ba}(N)$, and
Proposition~\ref{prop:weighted-fourier} with $A=4$ bounds the
frequencies $\delta\leq|\theta|\leq1/2$. Consequently,
\begin{equation}\label{eq:weighted-difference-l2}
 \E_{h_1\in[H]}
       \|\Delta\rho^{\rm wt}_{(h_1,\bh')}\|_{\ell^2}^2
 \ll_{k,\ell,\ba}\delta^3+L^{-4}
 \ll_{k,\ell,\ba}L^{-6/5}.
\end{equation}
Let $R=R_N^{\rm wt}$ be as in Lemma~\ref{lem:rho-tail}, so
$R\asymp_{k,\ell}L$. On $0\leq r\leq R$, Cauchy--Schwarz
in $r$ again gives
\begin{align*}
 &\E_{h_1\in[H]}\sum_{r=0}^{R}
                    |\Delta\rho^{\rm wt}_{(h_1,\bh')}(r)|\\
 &\quad\leq\sqrt{R+1}\,
       \E_{h_1\in[H]}
                    \|\Delta\rho^{\rm wt}_{(h_1,\bh')}\|_{\ell^2}\\
 &\quad\leq\sqrt{R+1}
       \left(\E_{h_1\in[H]}
                    \|\Delta\rho^{\rm wt}_{(h_1,\bh')}\|_{\ell^2}^2
       \right)^{1/2}
 \ll_{k,\ell,\ba}L^{1/2}L^{-3/5}=L^{-1/10}.
\end{align*}
For every $h_1\in[H]$, the tail estimate
\eqref{eq:rho-tail-log} implies
\begin{align*}
 \sum_{r>R}|\Delta\rho^{\rm wt}_{(h_1,\bh')}(r)|
 &\leq\sum_{r>R}\rho^{\rm wt}_{(h_1,\bh')}(r)
       +\sum_{r>R}\rho^{\rm wt}_{(h_1,\bh')}(r-1)\\
 &\leq2\sum_{s>R-1}\rho^{\rm wt}_{(h_1,\bh')}(s)
 \ll_{k,\ell,\ba}(\log N)^{-2}.
\end{align*}
Averaging the tail estimate over
$h_1$ and combining it with the previous estimate yields
\[
 D^{\rm wt}_{N,H}(\bh')
 \ll_{k,\ell,\ba}L^{-1/10}+(\log N)^{-2}
 \ll_{k,\ell,\ba}L^{-1/10}.
\]
This is uniform in $H$ and $\bh'$, as asserted.
\end{proof}

\section{Proof of  Theorems~\ref{thm:integers}- \ref{thm:weighted} and Corollary~\ref{cor:HL}}\label{sec:dynamics}

\begin{lemma}\label{lem:ergodic}
Let $(X, \nu, T)$ be a  uniquely ergodic system. Let $f\in C(X)$.
Let $I$ be a nonempty finite set and let
$\{\rho_\iota:\iota\in I\}$ be a family of finitely supported
nonnegative finite measures on $\mathbb Z_{\geq0}$ of a common mass
$m\geq0$, extended by zero
to $\mathbb Z$. Define
\[
 \mathcal A_\iota g=\sum_{r\geq0}\rho_\iota(r)g\circ T^r,
 \qquad
 d_\iota=\sum_{r\in\mathbb Z}|\rho_\iota(r)-\rho_\iota(r-1)|.
\]
For an integer $J\geq1$, put
$f_J=J^{-1}\sum_{j=0}^{J-1}f\circ T^j$.
Then
\begin{equation}\label{eq:smoothing}
 \E_{\iota\in I}\norm{\mathcal A_\iota f-m\nu(f)}_\infty
 \leq m\norm{f_J-\nu(f)}_\infty
       +\frac{J-1}{2}\norm f_\infty\,\E_{\iota\in I}d_\iota.
\end{equation}
\end{lemma}

\begin{proof}
Fix $\iota$ and abbreviate
$\rho=\rho_\iota$, $\mathcal A=\mathcal A_\iota$, and $d=d_\iota$.
Every sum defining $\mathcal A$ is finite, so the operator is linear
and maps $C(X)$ into $C(X)$.

Because $\rho(r)\geq0$ and $\sum_{r\geq0}\rho(r)=m$, for
$g\in C(X)$ and $x\in X$,
\begin{equation}\label{smoothing_trivial_bound}
 |\mathcal Ag(x)|
 \leq\sum_{r\geq0}\rho(r)|g(T^rx)|
 \leq\norm g_\infty\sum_{r\geq0}\rho(r)=m\norm g_\infty.
\end{equation}
Thus $\norm{\mathcal Ag}_\infty\leq m\norm g_\infty$ and
$\mathcal A\one=m$.  For $g\in C(X)$, replace $r+1$ by $s$ to obtain
\[
 \mathcal A(g\circ T)=\sum_{r\geq0}\rho(r)g\circ T^{r+1}
                     =\sum_{s\geq1}\rho(s-1)g\circ T^s.
\]
Subtracting $\mathcal Ag=\sum_{s\geq0}\rho(s)g\circ T^s$ gives
\[
 \mathcal A(g\circ T)-\mathcal Ag
 =\sum_{s\geq0}(\rho(s-1)-\rho(s))g\circ T^s.
\]
Then
\begin{align}
 \norm{\mathcal A(g\circ T)-\mathcal Ag}_\infty
 &\leq\sum_{s\geq0}|\rho(s-1)-\rho(s)|
                         \norm{g\circ T^s}_\infty \nonumber\\
 &\leq\norm g_\infty\sum_{s\geq0}|\rho(s-1)-\rho(s)|
 =\norm g_\infty d. \label{smoothing_telescoping_1}
\end{align}

Now, we use the technique of telescoping.
For an integer $j\geq1$, the telescoping identity gives
\[
 \mathcal A(f\circ T^j)-\mathcal Af
 =\sum_{i=0}^{j-1}
       \bigl(\mathcal A(f\circ T^{i+1})-\mathcal A(f\circ T^i)\bigr).
\]
Apply \eqref{smoothing_telescoping_1} to $g=f\circ T^i$. Because
$\norm{f\circ T^i}_\infty\leq\norm f_\infty$, we get
\begin{align*}
 \norm{\mathcal A(f\circ T^j)-\mathcal Af}_\infty
 &\leq\sum_{i=0}^{j-1}
       \norm{\mathcal A(f\circ T^{i+1})-\mathcal A(f\circ T^i)}_\infty\\
 &\leq\sum_{i=0}^{j-1}\norm f_\infty d
  =j\norm f_\infty d.
\end{align*}
For $j=0$ the difference is zero, so the same bound holds.
Linearity of $\mathcal A$ implies
\[
 \mathcal Af_J-\mathcal Af
 =\frac1J\sum_{j=0}^{J-1}
                  (\mathcal A(f\circ T^j)-\mathcal Af).
\]
Taking norms and using the preceding estimate, we obtain
\begin{align*}
 \norm{\mathcal Af_J-\mathcal Af}_\infty
 \leq\frac1J\sum_{j=0}^{J-1}j\norm f_\infty d
  =\frac{J-1}{2}\norm f_\infty d.
\end{align*}

By \eqref{smoothing_trivial_bound}, we have
\[
 \norm{\mathcal Af_J-m\nu(f)}_\infty
 =\norm{\mathcal A(f_J-\nu(f))}_\infty
 \leq m\norm{f_J-\nu(f)}_\infty.
\]
The triangle inequality now yields, for this fixed $\iota$,
\[
 \norm{\mathcal A_\iota f-m\nu(f)}_\infty
 \leq m\norm{f_J-\nu(f)}_\infty
       +\frac{J-1}{2}\norm f_\infty d_\iota.
\]
Then \eqref{eq:smoothing} follows by averaging over $\iota\in I$.
The argument also covers $m=0$, when every measure is zero. 
\end{proof}

We apply Lemma~\ref{lem:ergodic} with $m=1$ for the integer
and prime averages, and with $m=C_{\ba}(N)$ for the weighted
averages, to prove Theorems~\ref{thm:integers}, \ref{thm:primes},
and \ref{thm:weighted}. Since $X$ is a compact metric space and $(X,\nu,T)$ is uniquely ergodic, we have the following uniform ergodic theorem
\begin{equation}\label{eq:uniform-ergodicity}
 \norm{f_J-\nu(f)}_\infty\longrightarrow0\qquad(J\to\infty),
\end{equation}
see \cite[Theorem 6.19]{Walters1982}.

\begin{proof}[Proof of Theorem~\ref{thm:integers}]
Fix $k\geq1$, $f\in C(X)$, $\log N\leq H\leq N$, and
$\bh'\in[N]^{k-1}$. Apply Lemma~\ref{lem:ergodic} with
$I=[H]$, $\iota=h_1$, and
$\rho_\iota=\rho^{\rm int}_{(h_1,\bh')}$.
These are finitely supported probability measures on
$\mathbb Z_{\geq0}$, so their common mass is $m=1$.
For every $x\in X$, their averaging operators satisfy
\begin{equation}\label{average_to_density}
\begin{aligned}
 \sum_{r\geq0}\rho^{\rm int}_{(h_1,\bh')}(r)f(T^rx)
 &=\frac1N\sum_{n\leq N}\sum_{r\geq0}
       \one_{\{S_{(h_1,\bh')}(n)=r\}}f(T^rx)\\
 &=\frac1N\sum_{n\leq N}f(T^{S_{(h_1,\bh')}(n)}x)
 =A^{\rm int}_{N,(h_1,\bh')}f(x).
\end{aligned}
\end{equation}
The averaged difference norm in the lemma is
$D^{\rm int}_{N,H}(\bh')$. Thus, for every integer $J\geq1$,
\begin{align*}
 &\E_{h_1\in[H]}
       \norm{A^{\rm int}_{N,(h_1,\bh')}f-\nu(f)}_\infty \leq\norm{f_J-\nu(f)}_\infty
       +\frac{J-1}{2}\norm f_\infty D^{\rm int}_{N,H}(\bh').
\end{align*}
Use Proposition~\ref{prop:integer-invariance} and take the
suprema over $H$ and $\bh'$. We obtain
\begin{equation}\label{eq:integer-final-bound}
\begin{aligned}
 &\sup_{\log N\leq H\leq N}\sup_{\bh'\in[N]^{k-1}}
 \E_{h_1\in[H]}
       \norm{A^{\rm int}_{N,(h_1,\bh')}f-\nu(f)}_\infty\\
 &\hspace{12mm}\leq\norm{f_J-\nu(f)}_\infty
       +C_kJ\norm f_\infty L^{-1/5}\sqrt{\log L}.
\end{aligned}
\end{equation}

To justify the order of limits, let $\varepsilon>0$.
By \eqref{eq:uniform-ergodicity}, choose an integer $J$ such that
$\norm{f_J-\nu(f)}_\infty<\varepsilon/2$.
Keep this $J$ fixed. Since $L^{-1/5}\sqrt{\log L}\to0$,
there is an $N_0$ such that
\[
 C_kJ\norm f_\infty L^{-1/5}\sqrt{\log L}<\varepsilon/2
 \qquad(N\geq N_0).
\]
This choice of $N_0$ is independent of $H$, $\bh'$, and $x$.
Equation~\eqref{eq:integer-final-bound} then makes the supremum
of the averaged uniform-norm error less than $\varepsilon$,
proving \eqref{eq:integer-strong}.

For the absolute-value assertion, put
\[
 B^{\rm int}_{N,H,\bh'}(x)
 =\frac1{HN}\sum_{h_1=1}^{H}
       \abs{\sum_{n\leq N}f(T^{S_{(h_1,\bh')}(n)}x)}
 =\E_{h_1\in[H]}|A^{\rm int}_{N,(h_1,\bh')}f(x)|.
\]
The reverse triangle inequality gives
\begin{align*}
 \abs{B^{\rm int}_{N,H,\bh'}(x)-|\nu(f)|}
 &\leq\E_{h_1\in[H]}
       \abs{|A^{\rm int}_{N,(h_1,\bh')}f(x)|-|\nu(f)|}\\
 &\leq\E_{h_1\in[H]}
       \abs{A^{\rm int}_{N,(h_1,\bh')}f(x)-\nu(f)}\\
 &\leq\E_{h_1\in[H]}
       \norm{A^{\rm int}_{N,(h_1,\bh')}f-\nu(f)}_\infty.
\end{align*}
Taking the supremum over $x$, $H$, and $\bh'$ and applying
\eqref{eq:integer-strong} proves
\eqref{mainthm_dyn_Chowla_avg_eqn}.
\end{proof}

\begin{proof}[Proof of Theorem~\ref{thm:primes}]
Fix $k\geq1$, $f\in C(X)$, an allowed $H$, and
$\bh'\in[N]^{k-1}$. Apply Lemma~\ref{lem:ergodic} with
$I=[H]$, $\iota=h_1$, and
$\rho_\iota=\rho^{\rm pr}_{(h_1,\bh')}$. These are probability
measures, so their common mass is $m=1$. Expanding the definition,
\begin{align*}
 \sum_{r\geq0}\rho^{\rm pr}_{(h_1,\bh')}(r)f(T^rx)
 &=\frac1{\pi(N)}\sum_{p\leq N}\sum_{r\geq0}
       \one_{\{S_{(h_1,\bh')}(p)=r\}}f(T^rx)\\
 &=\frac1{\pi(N)}\sum_{p\leq N}f(T^{S_{(h_1,\bh')}(p)}x)
 =A^{\rm pr}_{N,(h_1,\bh')}f(x).
\end{align*}
Consequently, for every integer $J\geq1$,
\[
 \E_{h_1\in[H]}\norm{A^{\rm pr}_{N,(h_1,\bh')}f-\nu(f)}_\infty
 \leq\norm{f_J-\nu(f)}_\infty
       +\frac{J-1}{2}\norm f_\infty D^{\rm pr}_{N,H}(\bh').
\]
Proposition~\ref{prop:prime-invariance} gives
\begin{align}
 &\sup_{\exp(\sqrt{\log N})\leq H\leq N}
 \sup_{\bh'\in[N]^{k-1}}\E_{h_1\in[H]}
 \norm{A^{\rm pr}_{N,(h_1,\bh')}f-\nu(f)}_\infty\notag\\
 &\hspace{12mm}\leq\norm{f_J-\nu(f)}_\infty
                   +C_kJ\norm f_\infty L^{-1/10}.
 \label{eq:prime-final-bound}
\end{align}
Then all the uniformity assertions in Theorem~\ref{thm:primes} follow by the similar  argument to the proof of Theorem~\ref{thm:integers}.
\end{proof}

\begin{proof}[Proof of Theorem~\ref{thm:weighted}]
Fix $\exp(\sqrt{\log N})\leq H\leq N$ and
$\bh'\in[N]^{k-1}$. Use Lemma~\ref{lem:ergodic} with $I=[H]$,
$\iota=h_1$, $\rho_\iota=\rho^{\rm wt}_{(h_1,\bh')}$, and
$m=C_{\ba}(N)$. The measure has this mass for every $h_1$, and
\begin{align*}
 \sum_{r\geq0}\rho^{\rm wt}_{(h_1,\bh')}(r)f(T^rx)
 &=\frac1N\sum_{n=1}^{N}w_{\ba}(n)
       \sum_{r\geq0}\one_{\{S_{(h_1,\bh')}(n)=r\}}f(T^rx)\\
 &=\frac1N\sum_{n=1}^{N}w_{\ba}(n)f(T^{S_{(h_1,\bh')}(n)}x)
 =A^{\rm wt}_{N,(h_1,\bh')}f(x).
\end{align*}
The smoothing inequality is therefore
\begin{align*}
 &\E_{h_1\in[H]}
 \norm{A^{\rm wt}_{N,(h_1,\bh')}f-C_{\ba}(N)\nu(f)}_\infty\\
 &\qquad\leq C_{\ba}(N)\norm{f_J-\nu(f)}_\infty
       +\frac{J-1}{2}\norm f_\infty D^{\rm wt}_{N,H}(\bh').
\end{align*}
Proposition~\ref{prop:weighted-invariance} gives
\begin{align}
 &\sup_{\exp(\sqrt{\log N})\leq H\leq N}
 \sup_{\bh'\in[N]^{k-1}}\E_{h_1\in[H]}
 \norm{A^{\rm wt}_{N,(h_1,\bh')}f-C_{\ba}(N)\nu(f)}_\infty\notag\\
 &\hspace{12mm}\ll C_{\ba}(N)\norm{f_J-\nu(f)}_\infty
       +J\norm f_\infty L^{-1/10}.
 \label{eq:weighted-final-bound}
\end{align}
All the uniformity assertions in Theorem~\ref{thm:primes} follow by the similar  argument to the proof of Theorem~\ref{thm:integers}.
\end{proof}

\begin{proof}[Proof of Corollary~\ref{cor:HL}]
Under \eqref{eq:HL}, substitute
$C_{\ba}(N)=\mathfrak S(\ba)+o(1)$ into
\eqref{eq:factorization}. The error remains uniform in $H$,
$\bh'$, and $x$, since $\nu(f)$ is fixed.
If $\ba$ is not admissible, no dynamical estimate is needed:
for every $H\geq1$, every tuple of positive shifts $\bh'$, and
 every $x\in X$,
\[
 \abs{\frac1{HN}\sum_{h_1=1}^H\sum_{n=1}^{N}
            w_{\ba}(n)f(T^{S_{(h_1,\bh')}(n)}x)}
 \leq\norm f_\infty C_{\ba}(N)
 \ll_{\ell,\ba}\norm f_\infty\frac{(\log N)^{\ell+1}}N.
\]
Lemma~\ref{lem:mass} shows that this tends to zero without any
restriction on the growth of $H$ or $\bh'$.
\end{proof}

\section{Proof of Theorem~\ref{Chowla_conjecture_BR_form_conditional_result}} \label{sec_Chowla_BR_proof}

For $\Omega(n)$, we have an asymptotic formula for $\sum_{n\leq N} z^{\Omega(n)}$ by the work of Selberg \cite{Selberg1954}, which is stated below. It can be used to prove that $\Omega(n)$ satisfies the local Gaussian law, see \cite[Theorem 7.19]{MontgomeryVaughan2007}. 

\begin{theorem}[Sathe-Selberg  theorem]\label{Selberg1954}
Let $0<R<2$ be a real number. Then uniformly for $|z|\leq R$ and $N\ge2$, we have
\begin{equation}\label{Selberg1954_eqn}
	\sum_{n\leq N} z^{\Omega(n)} = G(z) N(\log N)^{z-1} +O\big(N (\log N)^{\Re z-2}\big),
\end{equation}
where $\Gamma(z)$ denotes the Gamma function and
$$G(z)= \frac1{\Gamma(z)} \prod_{p} \Big(1-\frac{z}{p}\Big)^{-1}\Big(1-\frac{1}{p}\Big)^z.$$
\end{theorem}

To apply Conjecture~\ref{conj_functional_chowla_str}, we show a $k$-dimensional generalization of Bergelson and Richter's theorem first.

\begin{theorem}\label{BergelsonRichter2022_k_dim}
	In a uniquely ergodic system $(X,\nu, T)$, we have
	\begin{equation}\label{BergelsonRichter2022_k_dim_eqn}
	\lim_{N\to\infty} \frac{1}{N^k}  \sum_{1\leq n_1,\dots,n_k\leq N} f(T^{\Omega(n_1\cdots n_k)}x) =\int_X f\,d\nu
\end{equation}
for any $x\in X$ and $f\in C(X)$.
\end{theorem}

\begin{proof} By Theorem~\ref{Selberg1954}, we have
\begin{equation*}
	\sum_{n\leq N} e(\theta\Omega(n)) \ll  N\exp\set{-8 \Vert \theta\Vert^2 \log\log N},
\end{equation*}
see \cite[Lemma A]{QiZheng2026} as well. Then
\begin{equation*}
	\sum_{1\leq n_1,\dots,n_k\leq N}  e(\theta\Omega(n_1\cdots n_k)) \ll  N^k\exp\set{-8 k \Vert \theta\Vert^2 \log\log N}.
\end{equation*}
This implies that for any $\ve>0$ and any $A>0$, we have
\begin{equation}\label{BergelsonRichter2022_k_dim_cond_A}
	\sum_{1\leq n_1,\dots,n_k\leq N}  e(\theta\Omega(n_1\cdots n_k)) \ll \frac{N^k}{(\log\log N)^A }
\end{equation}
for $\Vert \theta\Vert \ge\frac1{(\log\log N)^{\frac12-\ve}}$. 

On the partial sum of $\Omega(n)$, it is well known (e.g., \cite[Theorem 430]{HardyWright2008}) that
\begin{equation}\label{partial_sum_omega}
	\sum_{n\leq N} \Omega(n) = N\log\log N + O(N)
\end{equation}
for $N\ge2$. By \eqref{partial_sum_omega}, we have
\begin{equation}\label{BergelsonRichter2022_k_dim_cond_B}
	\sum_{1\leq n_1,\dots,n_k\leq N} \Omega(n_1\cdots n_k) =k N^{k-1} \sum_{n\leq N}\Omega(n) \ll_k N^k\log\log N.
\end{equation}

Thus, \eqref{BergelsonRichter2022_k_dim_eqn} follows by \eqref{BergelsonRichter2022_k_dim_cond_A}, \eqref{BergelsonRichter2022_k_dim_cond_B}, and  taking $P(Y_1,\dots,Y_k)=Y_1\cdots Y_k$ in \cite[Theorem A]{QiZheng2026}.
\end{proof}

\begin{proof}[Proof of Theorem~\ref{Chowla_conjecture_BR_form_conditional_result}] 
	
	Assume that Conjecture~\ref{conj_functional_chowla_str} is true. Let $[N]:=\set{1,2,\dots, N}$. Then for  any bounded $a\colon\N^k\to\C$, we have
	\begin{equation}\label{conj_functional_chowla_str_eqn}
		\BEu{n\in[N]} a\big(\Omega(n),\Omega(n+1),\dots,\Omega(n+k-1)\big)
= \BEu{(n_1,\ldots,n_k)\in[N]^k} a\big(\Omega(n_1),\dots,\Omega(n_k)\big) + \oh_{N\to\infty}(1).
	\end{equation}

 For any $x\in X$ and $f\in C(X)$, take $$a(n_1,\dots,n_k)= f(T^{n_1+\cdots +n_k}x).$$
	Then
	\[
a\big(\Omega(n),\Omega(n+1),\dots,\Omega(n+k-1)\big)
= f(T^{\Omega(n)+\Omega(n+1)+\dots+\Omega(n+k-1)}x)
	\]
	and 
	\[
a\big(\Omega(n_1),\dots,\Omega(n_k)\big) = f(T^{\Omega(n_1)+\cdots +\Omega(n_k)}x)= f(T^{\Omega(n_1\cdots n_k)}x).
	\]
	By \eqref{conj_functional_chowla_str_eqn} and Theorem~\ref{BergelsonRichter2022_k_dim}, we get that
\begin{equation*}
		\lim_{N\to\infty} \frac{1}{N}  \sum_{n\leq N} f(T^{\Omega(n)+\Omega(n+1)+\cdots + \Omega(n+k-1)}x) =\int_X f\,d\nu.
\end{equation*}
This completes the proof.
\end{proof}

Moreover, as a consequence of Theorem~\ref{BergelsonRichter2022_k_dim}, we prove a variant of \eqref{BergelsonRichter2022_k_dim_eqn}, which is of independent interest. 

\begin{theorem}\label{BergelsonRichter2022_lcm}
	In a uniquely ergodic system $(X,\nu, T)$, we have
	\begin{equation}\label{BergelsonRichter2022_lcm_eqn}
	\lim_{N\to\infty} \frac{1}{N^k}  \sum_{1\leq n_1,\dots,n_k\leq N} f(T^{\Omega([n_1,\cdots, n_k])}x) =\int_X f\,d\nu
\end{equation}
for any $x\in X$ and $f\in C(X)$, where $[n_1,\cdots, n_k]$ denotes the least common multiple of the integers $n_1,\cdots, n_k$.
\end{theorem}

\begin{proof}
	Let \(I:=\int_X f\,d\nu\) and \(\|f\|_\infty=M.\)
Fix \(w\geq2\), and put
\[
P_w:=\prod_{p\leq w}p.
\]

Every integer \(n\) has a unique decomposition
\[
n=a_w(n)b_w(n),
\]
where \(a_w(n)\) contains all the prime factors of \(n\) not exceeding \(w\) and $(b_w(n),P_w)=1$.
For \(n_1,\ldots,n_k\), we write \(a_i=a_w(n_i)\) and \(b_i=b_w(n_i)\), and let
\begin{align*}
	S_N&:= \frac{1}{N^k}  \sum_{1\leq n_1,\dots,n_k\leq N} f(T^{\Omega([n_1,\cdots, n_k])}x),\\
	S_{N,w}&:=\frac1{N^k}\sum_{n_1,\ldots,n_k\leq N}
f\!\left(T^{\Omega([a_1,\ldots,a_k])+\sum_{i=1}^k\Omega(b_i)}x\right).
\end{align*}

If the integers \(b_1,\ldots,b_k\) are pairwise coprime, then
\[
\Omega([n_1,\ldots,n_k])
=\Omega([a_1,\ldots,a_k])+\sum_{i=1}^k\Omega(b_i).
\]
The summands of $S_N$ and $S_{N,w}$ agree on such integers.

If the integers \(b_1,\ldots,b_k\) are not pairwise coprime, there exist some \(i<j\) and a prime \(p>w\) such that $p\mid n_i$ and \( p \mid n_j\), then
\begin{equation}\label{exceptional_tuples}
	\frac1{N^k}
\#\{(n_1,\ldots,n_k): b_i\text{ are not pairwise coprime}\}
\leq \binom{k}{2}\sum_{p>w}\frac1{p^2}.
\end{equation}

 Thus, by the triangle inequality and \eqref{exceptional_tuples}, we obtain that
\begin{equation}\label{difference_S_Nw_S_N}
	|S_{N,w}-S_N|
\leq 2M\binom{k}{2}\sum_{p>w}\frac1{p^2}.
\end{equation}

Now, we show that, for every fixed \(w\), 
\begin{equation}\label{lim_S_Nw}
	\lim_{N\to\infty}S_{N,w}= I. 
\end{equation}

Using
\[
\mathbf{1}_{(b_i, P_w)=1} = \sum_{d_i\mid (b_i, P_w)} \mu(d_i),
\]
where $\mu(n)$ is the M\"obius function, we have
\begin{align}
	S_{N,w} & = \frac1{N^k} \sum_{\substack{a_1,\dots,a_k \leq N \\ p\mid a_i\Rightarrow p\mid P_w, \forall i}} \sum_{\substack{b_i \leq \frac{N}{a_i}\\ (b_i, P_w)=1, \forall i}} f\!\left(T^{\Omega([a_1,\ldots,a_k])+\sum_{i=1}^k\Omega(b_i)}x\right) \nonumber \\
	& = \frac1{N^k} \sum_{\substack{a_1,\dots,a_k \leq N \\ p\mid a_i\Rightarrow p\mid P_w, \forall i}} \sum_{d_1,\dots,d_k \mid P_w} \mu(d_1)\cdots \mu(d_k) \sum_{b_i \leq \frac{N}{a_id_i}, \forall i} f\!\left(T^{\Omega([a_1,\ldots,a_k])+\sum_{i=1}^k\Omega(d_i)+\sum_{i=1}^k\Omega(b_i)}x\right) \nonumber\\
	& = \sum_{d_1,\dots,d_k \mid P_w} \frac{\mu(d_1)\cdots \mu(d_k)}{d_1\cdots d_k}  \sum_{\substack{a_1,\dots,a_k \leq N \\ p\mid a_i\Rightarrow p\mid P_w, \forall i}}   \frac1{a_1\cdots a_k}\cdot \nonumber\\
	&\qquad\cdot \frac1{N^k/(a_1d_1\cdots a_kd_k)} \sum_{b_i \leq \frac{N}{a_id_i}, \forall i} f\!\left(T^{\Omega([a_1,\ldots,a_k])+\sum_{i=1}^k\Omega(d_i)+\sum_{i=1}^k\Omega(b_i)}x\right).
\end{align}

Notice that
\[
\sum_{\substack{a\geq1\\p\mid a\Rightarrow p\mid P_w}}\frac1a
=\prod_{p\mid P_w}\left(1-\frac1p\right)^{-1}
=\frac{P_w}{\varphi(P_w)},
\]
and
\[
\sum_{\substack{a>A\\p\mid a\Rightarrow p\mid P}}\frac1a =\oh_{A\to\infty}(1)
\]
for $A\ge1$.
We may divide the summation $S_{N,w}$ into two parts
\begin{multline}\label{S_Nw_key_identity}
S_{N,w} =	\sum_{d_1,\dots,d_k \mid P_w} \frac{\mu(d_1)\cdots \mu(d_k)}{d_1\cdots d_k}  \sum_{\substack{a_1,\dots,a_k \leq A \\ p\mid a_i\Rightarrow p\mid P_w, \forall i}}   \frac1{a_1\cdots a_k}\cdot \\
	\cdot \frac1{N^k/(a_1d_1\cdots a_kd_k)} \sum_{b_i \leq \frac{N}{a_id_i}, \forall i} f\!\left(T^{\Omega([a_1,\ldots,a_k])+\sum_{i=1}^k\Omega(d_i)+\sum_{i=1}^k\Omega(b_i)}x\right)+ \oh_{A\to\infty}(1)
\end{multline}
for any $A\leq N$. By Theorem~\ref{BergelsonRichter2022_k_dim}, taking $N\to\infty$ in \eqref{S_Nw_key_identity} gives
\begin{equation}\label{S_Nw_key_identity_1}
	\lim_{N\to\infty} S_{N,w} = I\cdot \sum_{d_1,\dots,d_k \mid P_w} \frac{\mu(d_1)\cdots \mu(d_k)}{d_1\cdots d_k}  \sum_{\substack{a_1,\dots,a_k \leq A \\ p\mid a_i\Rightarrow p\mid P_w, \forall i}}   \frac1{a_1\cdots a_k} + \oh_{A\to\infty}(1).
\end{equation}
Then taking $A\to\infty$ in \eqref{S_Nw_key_identity_1} gives
\begin{equation*}
	\lim_{N\to\infty} S_{N,w} = I\cdot \sum_{d_1,\dots,d_k \mid P_w} \frac{\mu(d_1)\cdots \mu(d_k)}{d_1\cdots d_k} \cdot \Big(\frac{P_w}{\varphi(P_w)} \Big)^k = I\cdot \Big(\frac{\varphi(P_w)}{P_w} \Big)^k \cdot \Big(\frac{P_w}{\varphi(P_w)} \Big)^k =I.
\end{equation*}

Thus, by \eqref{difference_S_Nw_S_N} and \eqref{lim_S_Nw} we obtain that
\[
\limsup_{N\to\infty}|S_N-I|
\leq 2M\binom{k}{2}\sum_{p>w}\frac1{p^2}.
\]
Letting \(w\to\infty\), the right-hand side tends to zero, we get \eqref{BergelsonRichter2022_lcm_eqn}, as desired. 
\end{proof}

\begin{remark}
Let $k\ge2$.	For any set $S$ of natural numbers, one can use the method in \cite{DengWang2026} to prove that 
	 \begin{equation*}
	\lim_{N\to\infty} \frac{1}{N^k}  \sum_{\substack{1\leq n_1,\dots,n_k\leq N \\ \gcd(n_1,\dots,n_k)\in S}} f(T^{\Omega([n_1,\cdots, n_k])}x) =\frac{\zeta_S(k)}{\zeta(k)}\int_X f\,d\nu,
\end{equation*}
where $\zeta_S(k)=\sum_{n\in S}\frac1{n^k}.$
\end{remark}

\section{Proof of Theorems~\ref{Charamaras_Richter_conj_log_avg} and \ref{Tao2016_dyn}}\label{sec_Tao2016_dyn}

Write $L:=L_N=\log\log N$ for short in this section. Let $\Nzero$ be the set of nonnegative integers. For an integer $N\geq 1$, define the probability measure $\overline{\pi}_N$ on
$\Nzero$ by
\[
\overline{\pi}_N(\ell):=\frac{1}{N}\#\{n\leq N:\Omega(n)=\ell\}.
\]
Define 
\begin{equation*}
	\overline{\pi}_N(k,\ell):=\frac1{\log N}\sum_{\substack{n\leq N\\
              \Omega(n)=k,\ \Omega(n+1)=\ell}}\frac1n,
\end{equation*}
and
\begin{equation*}\label{eq:d-def}
 d_N(k,\ell):=\overline{\pi}_N(k,\ell)-\overline{\pi}_N(k)\overline{\pi}_N(\ell).
\end{equation*}

Then we may rewrite \eqref{CharamarasRichter2025_thm_A} in the following bilinear form with coefficients $\{d_N(k,\ell): k,\ell\ge0\}$.

\begin{proposition}\label{prop:CR}
For all sufficiently large $N$,
\begin{equation}\label{eq:product-norm}
 \sup_{\substack{\norm{u}_\infty\leq1\\\norm{v}_\infty\leq1}}
 \abs{\sum_{k,\ell\geq0}d_N(k,\ell)u(k)v(\ell)}
 \ll L^{-1/2}.
\end{equation}
\end{proposition}

To apply Proposition~\ref{prop:CR}, we cite a finite complex form of Grothendieck's inequality
\cite{Grothendieck}, see also \cite[Eq. (1) and Lemma 2.2]{FriedlandLimZhang2018}.

\begin{theorem}[Grothendieck's inequality]\label{thm:GI}
Let $\mathcal H$ be a complex Hilbert space. Then there is an absolute positive constant $C$ such that every finite complex matrix
$M=(m_{ij})$ satisfies
\begin{equation}\label{eq:GI}
 \sup_{\substack{x_i,y_j\in\mathcal H\\
                  \norm{x_i}\leq1,\ \norm{y_j}\leq1}}
 \abs{\sum_{i,j}m_{ij}\langle x_i,y_j\rangle}
 \leq C
 \sup_{\substack{\abs{s_i}\leq1\\\abs{t_j}\leq1}}
 \abs{\sum_{i,j}m_{ij}s_it_j}.
\end{equation}
\end{theorem}

In the following, we use Theorem~\ref{thm:GI} to estimate the bilinear forms related to \eqref{eq:product-norm}.

\begin{proposition}\label{prop:upgrade}
Let $X$ be a countable space. Let $P$ be a probability measure on a product $X\times X$, let $Q$
be a probability measure on $X$, and put
$D(i,j)=P(i,j)-Q(i)Q(j)$.  Suppose that
\begin{equation}\label{eq:eta-assumption}
 \sup_{\substack{\norm{u}_\infty\leq1\\\norm{v}_\infty\leq1}}
 \abs{\sum_{i,j}D(i,j)u(i)v(j)}\leq\eta.
\end{equation}
If $T\subset X$ is finite, $\abs{T}=d$, and $Q(T^c)\leq\tau$, then
there exists some positive constant $C$ such that 
\begin{equation}\label{eq:upgrade}
 \abs{\sum_{i,j}D(i,j)A(i,j)}
 \leq C\norm{A}_\infty\bigl(\tau+\eta+\sqrt d\,\eta\bigr),
\end{equation}
for any bounded function of two variables $A\colon X\times X\to\C$.
\end{proposition}

\begin{proof}
We may assume that $\norm{A}_\infty\leq1$.  Taking
$u=\one_{T^c}$ and $v=1$ in \eqref{eq:eta-assumption} gives
\[
 P(T^c\times X)\leq Q(T^c)+\eta\leq\tau+\eta.
\]
Similarly, $P(X\times T^c)\leq\tau+\eta$.  Hence, with
$S=(T\times T)^c$, we have $S\subseteq (T^c\times X)\cup (X\times T^c)$ and hence
\begin{equation*}\label{eq:tails-P}
 P(S)\leq2\tau+2\eta,
 \qquad (Q\otimes Q)(S)\leq2\tau.
\end{equation*}
Here $(Q\otimes Q)(S):=\sum_{(i,j)\in S}Q(i)Q(j)$. Consequently,
\begin{equation}\label{eq:tail-D}
 \abs{\sum_{(i,j)\in S}D(i,j)A(i,j)}
 \leq \sum_{(i,j)\in S}|D(i,j)| \leq  P(S)+(Q\otimes Q)(S)\leq4\tau+2\eta.
\end{equation}

It remains to estimate the sum on $T\times T$.  Enumerate $T$ and regard
each row $(A(i,j))_{j\in T}$ as a vector in $\C^d$. Set
\[
 x_i:=d^{-1/2}(A(i,j))_{j\in T},
 \qquad y_j:=e_j \quad (i,j\in T),
\]
where $(e_j)_{j\in T}$ is the standard orthonormal basis.  Then
$\norm{x_i},\norm{y_j}\leq1$ and
\begin{equation*}\label{eq:factor-A}
 A(i,j)=\sqrt d\,\langle x_i,y_j\rangle.
\end{equation*}
Apply \cref{thm:GI} to the restricted matrix
$(D(i,j))_{i,j\in T}$.  Scalar sequences on $T$ can be extended by zero to
$X$, so \eqref{eq:eta-assumption} bounds the scalar supremum in
\eqref{eq:GI} by $\eta$.  Therefore, there exists some positive constant $C$ such that
\begin{equation}\label{eq:central-D}
 \abs{\sum_{i,j\in T}D(i,j)A(i,j)}
 \leq C\sqrt d\,\eta.
\end{equation}
Combining \eqref{eq:tail-D} and \eqref{eq:central-D} proves the result.
\end{proof}

\begin{proof}[Proof of Theorems~\ref{Charamaras_Richter_conj_log_avg} and \ref{Tao2016_dyn}]
Assume that $\norm{a}_\infty\leq1$.  Let $R>0$ be a parameter to be chosen later, and let $T=\{k\in\Nzero:\abs{k-L}\leq R\sqrt L\}$. By the Tur\'an inequality (e.g.,  \cite[Theorem 3.1.2]{CojocaruMurty2006}), we have
\begin{equation}\label{eq:turan-kubilius}
  \frac{1}{N}\sum_{n\leq N}
    |\Omega(n)-L|^2
  \ll L.
\end{equation}
Then Chebyshev's inequality gives us that
\begin{equation}\label{eq:tau-choice}
 \overline{\pi}_N(T^c)= \frac1N \sum_{\substack{n\leq N \\ |\Omega(n)-L|>R\sqrt{L}}} 1 \leq \frac1N  \sum_{n\leq N} \frac{|\Omega(n)-L|^2}{R^2L}\ll R^{-2}.
\end{equation}
Moreover, set
\begin{equation}\label{eq:d-choice}
 d:=\abs T\ll R\sqrt L.
\end{equation}
Apply Proposition~\ref{prop:upgrade} with $P(k,\ell)=\overline{\pi}_N(k,\ell)$, $Q(\ell)=\overline{\pi}_N(\ell)$, and
$\eta\ll L^{-1/2}$ supplied by Proposition~\ref{prop:CR}.  Then Equations
\eqref{eq:tau-choice} and \eqref{eq:d-choice} give
\begin{equation}\label{eqn_dN_total}
	\abs{\sum_{k,\ell\geq0}d_N(k,\ell)a(k,\ell)}\ll R^{-2} +L^{-1/2}+\sqrt{R\sqrt L} \cdot L^{-1/2} = R^{-2} +L^{-1/2}+ R^{1/2} L^{-1/4}.
\end{equation}
Take $R^{-2} =  R^{1/2} L^{-1/4}$, i.e., $R=L^{1/{10}}$.
By the definition of $d_N(k,\ell)$, from \eqref{eqn_dN_total} we obtain that
\begin{equation}\label{eq:HN-version}
 \abs{
 \frac1{\log N}\sum_{n\leq N}\frac{a(\Omega(n),\Omega(n+1))}{n}
 -\frac1{N^2}\sum_{n_1,n_2\leq N}
 a(\Omega(n_1),\Omega(n_2))}
 \ll L^{-1/5},
\end{equation}
which completes the proof of Theorem~\ref{Charamaras_Richter_conj_log_avg}.

Similar to the proof of Theorem~\ref{Chowla_conjecture_BR_form_conditional_result}, Theorem~\ref{Tao2016_dyn} follows by combining Theorems~\ref{Charamaras_Richter_conj_log_avg} and \ref{BergelsonRichter2022_k_dim}. 
\end{proof}

\section{Proof of Theorems~\ref{thm:main_sq} and~\ref{thm:one-point}}\label{sec_sq_proofs}

\subsection{Shifted invariance on the local density of the Omega function}\label{sec_Shifted_invariance_1}

For $h\in \Z$, we define
\[
 \Delta_N(h):=
 \sum_{\ell\in\Z}
 |\overline{\pi}_N(\ell+h)-\overline{\pi}_N(\ell)|.
\]
In \cite[Remark 1.3]{BergelsonRichter2022}, Kanigowski and Radziwi\l\l{} proposed an approach to prove that $\Delta_N(1)=o_{N\to\infty}(1)$. In the following, using their method, we prove an upper bound for $\Delta_N(1)$.

\begin{theorem}\label{lem:translation}
For the unit translation, we have
\[
 \Delta_N(1):=
 \sum_{\ell\in\Z}
 |\overline{\pi}_N(\ell+1)-\overline{\pi}_N(\ell)|
 \ll(\log\log N)^{-1/2}.
\]
\end{theorem}

\begin{proof}
Put $L=L_N:=\log\log N$ and recall that
\[
 G_L(\ell)=\frac1{\sqrt{2\pi L}}
 \exp\left(-\frac{(\ell-L)^2}{2L}\right),
 \qquad \ell\in\Z.
\]
By Theorem~\ref{thm_strong-local-EK} again, we have
\begin{equation}\label{eq:local-gaussian}
 \overline{\pi}_N(\ell)
 =G_L(\ell)
 \left(1+O\left(L^{-1/2}+\frac{|\ell-L|^3}{L^2}\right)\right),
\end{equation}
uniformly for positive integers $\ell\leq\tfrac32L$. By the quantitative
Erd\H{o}s--Kac theorem of R\'enyi and Tur\'an \cite{RenyiTuran1958}, we have
\begin{equation}\label{eq:quantitative-EK}
 \sup_{t\in\mathbb R}
 \Big|
  \sum_{\ell\leq L+t\sqrt L}\overline{\pi}_N(\ell)-\Phi(t)
 \Big|
 \ll L^{-1/2},
\end{equation}
where
\[
 \Phi(t):=\frac1{\sqrt{2\pi}}
 \int_{-\infty}^t e^{-u^2/2}\,du.
\]

Set $A:=\log L$ and the typical range
\[
 T_{A,N}:=\{\ell\in\Z:|\ell-L|\leq A\sqrt L\}.
\]
Since $A=o(\sqrt L)$, the set $T_{A,N}$ lies in the range of
\eqref{eq:local-gaussian} for all sufficiently large $N$.  Moreover,
by Lemma~\ref{lem_lattice_moments},
\begin{align*}
 \sum_{\ell\in\Z}G_L(\ell)&\ll1,\\
 \sum_{\ell\in\Z}G_L(\ell)|\ell-L|^3&\ll L^{3/2}.
\end{align*}
Thus \eqref{eq:local-gaussian}, summed with its Gaussian weight rather
than estimated by its largest error, yields
\begin{equation}\label{eq:local-l1}
 \sum_{\ell\in T_{A,N}}
 |\overline{\pi}_N(\ell)-G_L(\ell)|
 \ll L^{-1/2}.
\end{equation}

The sequence $(G_L(\ell))_{\ell\in\Z}$ is unimodal and
$\max_\ell G_L(\ell)\ll L^{-1/2}$.  Telescoping along the two monotone
sides of this sequence therefore gives
\begin{equation}\label{eq:gaussian-shift}
 \sum_{\ell\in\Z}|G_L(\ell+1)-G_L(\ell)|
 \leq2\max_\ell G_L(\ell)
 \ll L^{-1/2}.
\end{equation}

Let
\[
 \mathcal C_N:=
 \{\ell\in\Z:|\ell-L|\leq(A-1)\sqrt L\}.
\]
For sufficiently large $N$, $\ell\in\mathcal C_N$ implies
$\ell,\ell+1\in T_{A,N}$.  It follows from
\eqref{eq:local-l1} and \eqref{eq:gaussian-shift} that
\begin{multline}\label{eq:central-shift}
	\sum_{\ell\in\mathcal C_N}
 |\overline{\pi}_N(\ell+1)-\overline{\pi}_N(\ell)|\leq
 \sum_{\ell\in\Z}|G_L(\ell+1)-G_L(\ell)|
 \\+\sum_{\ell\in\mathcal C_N}
  |\overline{\pi}_N(\ell+1)-G_L(\ell+1)|
 +\sum_{\ell\in\mathcal C_N}
  |\overline{\pi}_N(\ell)-G_L(\ell)|
 \ll L^{-1/2}.
\end{multline}
	
The standard Gaussian tail bound and \eqref{eq:quantitative-EK} give
\begin{equation}\label{eq:EK-tail}
 \sum_{\ell\notin\mathcal C_N}\overline{\pi}_N(\ell)
 \ll 1-\Phi(A-1)+L^{-1/2}\ll L^{-1/2},
\end{equation}
see the proof of \cite[Lemma 2.3]{CharamarasRichter2025}.
After changing variables, the same estimate applies to
$\sum_{\ell\notin\mathcal C_N}\overline{\pi}_N(\ell+1)$.  Hence, by nonnegativity,
\[
 \sum_{\ell\notin\mathcal C_N}
 |\overline{\pi}_N(\ell+1)-\overline{\pi}_N(\ell)|
 \ll L^{-1/2}.
\]
Combining this with \eqref{eq:central-shift} proves the theorem.
\end{proof}

\begin{remark}
	We remark that, similar to the proof of \cite[(2.4)]{CharamarasRichter2025}, one may use the Fundamental Theorem of Calculus to  prove \eqref{eq:gaussian-shift}. In the proof of Theorem~\ref{lem:translation}, the unimodality of $G_L(\ell)$ simplifies the proof of \eqref{eq:gaussian-shift}.
\end{remark}

In \cite{Richter2021}, Richter proved that for any bounded function $a:\N\to\R$ we have 
\begin{equation}\label{Richter2021}
	\frac{1}{N}\sum_{n=1}^N a(\Omega(n)+1) = 	\frac{1}{N}\sum_{n=1}^N a(\Omega(n)) + o_{N\to\infty}(1).
\end{equation} 
This identity can be used to prove \eqref{BergelsonRichter2022}. Using Theorem~\ref{lem:translation}, we obtain Richter's result \eqref{Richter2021} with an error term.

\begin{theorem}\label{Richter2021_err_term}
	Let $a:\N\to \C$ be bounded. Then we have
	\[
	\frac1{ N}\sum_{n=1}^N a(\Omega(n)+1)= \frac1{ N}\sum_{n=1}^N a(\Omega(n))+ O\Big( \frac{1}{\sqrt{\log\log N}} \Big).
	\]
\end{theorem}

\begin{proof}
We may assume that $|a|\leq1$. Let $\P_\ell=\set{n\in \N: \Omega(n)=\ell}$. Then
\begin{align*}
 \Esub{n\in[N]} a(\Omega(n)) &= \frac{1}{N} \sum_{n\leq N}  a(\Omega(n)) = \frac{1}{N} \sum_{n\leq N}  \sum_{\ell=1}^\infty a(\ell) \1_{\ell=\Omega(n)} =   \frac{1}{N}\sum_{n\leq N}  \sum_{\ell=1}^\infty a(\ell)\1_{\P_\ell}(n) \\
&= \sum_{\ell=1}^\infty a(\ell) \cdot  \frac{1}{N} \sum_{n\leq N}  \1_{\P_\ell}(n)  = \sum_{\ell=1}^\infty a(\ell)\overline{\pi}_N(\ell),
\end{align*}
and
\begin{align*}
 \Esub{n\in[N]} a(\Omega(n)+1) &= \frac{1}{N} \sum_{n\leq N}  a(\Omega(n)+1)= \frac{1}{N} \sum_{n\leq N}  \sum_{\ell=1}^\infty a(\ell+1) \1_{\ell=\Omega(n)} =   \frac{1}{N}\sum_{n\leq N}  \sum_{\ell=1}^\infty a(\ell+1)\1_{\P_\ell}(n) \\
&= \sum_{\ell=1}^\infty a(\ell+1) \cdot  \frac{1}{N} \sum_{n\leq N}  \1_{\P_\ell}(n) = \sum_{\ell=1}^\infty a(\ell+1)\overline{\pi}_N(\ell).
\end{align*}
Write
\[
\Esub{n\in[N]} a(\Omega(n)) = a(1) \overline{\pi}_N(1) + \sum_{\ell=1}^\infty a(\ell+1)\overline{\pi}_N(\ell+1).
\]
Then by the triangle inequality and $|a|\leq1$, we have
\begin{align*}
\Big|  \Esub{n\in[N]} a(\Omega(n)+1) -  \Esub{n\in[N]} a(\Omega(n)) \Big| &=  \Big|  a(1) \overline{\pi}_N(1)  + \sum_{\ell=1}^\infty a(\ell+1) \big(\overline{\pi}_N(\ell+1) - \overline{\pi}_N(\ell)\big) \Big| \\
&\leq |a(1)| \overline{\pi}_N(1)  + \sum_{\ell=1}^\infty |a(\ell+1)| |\overline{\pi}_N(\ell+1) - \overline{\pi}_N(\ell)| \\
&\leq \overline{\pi}_N(1)  + \sum_{\ell=1}^\infty |\overline{\pi}_N(\ell+1) - \overline{\pi}_N(\ell)|.
\end{align*}

By the prime number theorem, 
\[
 \overline{\pi}_N(1)   \ll \frac1{\log N}.
\]
Then the desired estimate follows by Theorem~\ref{lem:translation}.
\end{proof}

\subsection{Averages over complete digit blocks}\label{sec:digit-blocks}

For $r\geq1$, define
\[
 \mathcal R_rf(y):=\frac1{q^r}\sum_{0\leq b<q^r}
 f\bigl(T^{s_q(b)}y\bigr),
 \qquad y\in X.
\]

\begin{lemma}\label{lem:digit-block}
For every $f\in C(X)$,
\[
 \lim_{r\to\infty}\sup_{y\in X}
 \Big|\mathcal R_rf(y)-\int_Xf\,d\nu\Big|=0.
\]
\end{lemma}

\begin{proof}
Put
\[
 \rho_r(\ell):=\frac1{q^r}
 \#\{0\leq b<q^r:s_q(b)=\ell\},
 \qquad \ell\in\Z,
\]
where $\rho_r(\ell)=0$ outside $0\leq\ell\leq r(q-1)$.    Then
\begin{equation}\label{eq:R-rho}
 \mathcal R_r f(y)=
 \sum_{\ell\in\Z}\rho_r(\ell)f(T^\ell y).
\end{equation}

Let
\[
 u_q(d):=\frac1q\mathbf 1_{\{0,\ldots,q-1\}}(d).
\]
Then $\rho_r=u_q^{\ast r}$, where $u_q^{\ast r}$ denotes the $r$th additive convolution of $u_q$ on $\Z$. Then by \cite[(5) and (8)]{MattnerRoos2008}, $\rho_r$ is symmetric, unimodal and satisfies
\begin{equation*}
	\max_{\ell\in \Z}\rho_r(\ell) < \frac{2\sqrt{2/\pi}}{q\sqrt r}.
\end{equation*}
This implies that
\begin{equation}\label{eq:max-rho}
 \max_{\ell\in\Z}\rho_r(\ell)\ll_q \frac1{\sqrt{r}}.
\end{equation}

Since $\rho_r$ is unimodal, we have
\[
 \sum_{\ell\in\Z}|\rho_r(\ell-1)-\rho_r(\ell)|
 =2\max_{\ell\in\Z}\rho_r(\ell).
\]
For every integer $j\geq1$, the pointwise telescoping inequality gives
\[
 |\rho_r(\ell-j)-\rho_r(\ell)|
 \leq\sum_{h=0}^{j-1}
 |\rho_r(\ell-h-1)-\rho_r(\ell-h)|.
\]
Summing over $\ell$, changing variables in each inner sum, and using
\eqref{eq:max-rho} gives
\begin{equation}\label{eq:rho-shift}
	\sum_{\ell\in\Z}|\rho_r(\ell-j)-\rho_r(\ell)|\leq j\sum_{\ell\in\Z}
 |\rho_r(\ell-1)-\rho_r(\ell)|\ll_q\frac{j}{\sqrt r}.                            
\end{equation}
The same estimate is immediate for $j=0$. Let
\[
 I_f:=\int_Xf\,d\nu,
 \qquad
 E_Hf(z):=\frac1H\sum_{j=0}^{H-1}f(T^jz),
\]
and
\[
\varepsilon_H(f):=
 \sup_{z\in X}|E_Hf(z)-I_f|.
\]
Since $X$ is a compact metric space and $(X,\nu,T)$ is uniquely ergodic, by \eqref{uniform_ergodicity} we have 
\begin{equation}\label{eq:uniform-ergodic}
 \lim_{H\to\infty}\varepsilon_H(f)=0
\end{equation}
Define
\[
 \mathcal A_{r,H}f(y):=
 \sum_{\ell\in\Z}\rho_r(\ell)E_Hf(T^\ell y).
\]
Expanding $E_H$ and then making the change of variables
$m=\ell+j$, we find
\begin{align}
 \mathcal A_{r,H}f(y)
 &=\frac1H\sum_{j=0}^{H-1}\sum_{\ell\in\Z}
 \rho_r(\ell)f(T^{\ell+j}y)=\frac1H\sum_{j=0}^{H-1}\sum_{m\in\Z}
 \rho_r(m-j)f(T^my).                                  \label{eq:A-expanded}
\end{align}
All these sums are finite because $\rho_r$ has finite support.  On the
other hand, by \eqref{eq:R-rho},
\[
 \mathcal R_r f(y)
 =\frac1H\sum_{j=0}^{H-1}\sum_{m\in\Z}
 \rho_r(m)f(T^my).
\]
Subtracting this identity from \eqref{eq:A-expanded}, it follows by the
triangle inequality and \eqref{eq:rho-shift} that
\begin{align}
 |\mathcal R_r f(y)-\mathcal A_{r,H}f(y)|
 &\leq\frac{\norm{f}_\infty}{H}
 \sum_{j=0}^{H-1}\sum_{m\in\Z}
 |\rho_r(m-j)-\rho_r(m)|\ll_q\frac{\norm{f}_\infty}{H\sqrt r}
 \sum_{j=0}^{H-1}j\ll_{q,f} \frac{H}{\sqrt r}.           \label{eq:R-A}
\end{align}

Finally, by $\rho_r(\ell)\geq0$ and
$\sum_{\ell\in\Z}\rho_r(\ell)=1$, we have
\begin{align}
 |\mathcal A_{r,H}f(y)-I_f|
 &=\Big|\sum_{\ell\in\Z}\rho_r(\ell)
 \bigl(E_Hf(T^\ell y)-I_f\bigr)\Big|\notag\\
 &\leq\sum_{\ell\in\Z}\rho_r(\ell)
 |E_Hf(T^\ell y)-I_f|\notag\\
 &\leq\varepsilon_H(f)
 \sum_{\ell\in\Z}\rho_r(\ell)
 =\varepsilon_H(f). \label{eq:A-I}
\end{align}
 Combining
\eqref{eq:R-A} and \eqref{eq:A-I}, and then taking the supremum over
$y\in X$, gives
\[
 \sup_{y\in X}|\mathcal R_r f(y)-I_f|
 \leq\varepsilon_H(f)
 +O_{q,f}\left(\frac{H}{\sqrt r}\right).
\]
First let $r\to\infty$ with $H$ fixed, and then let $H\to\infty$.
Then the lemma follows by \eqref{eq:uniform-ergodic}.
\end{proof}

Now, we use Lemma~\ref{lem:digit-block} to prove the uniform convergence in Theorem~\ref{thm:main_sq}.

\begin{proof}[Proof of Theorem~\ref{thm:main_sq}]
Fix $r\geq1$ and put $Q=q^r$.  Write
\[
 N=MQ+R,
 \qquad M\geq0,
 \qquad 0\leq R<Q.
\]
If $0\leq b<Q$, by the additivity of $s_q(n)$, we have
\begin{equation}\label{eq:digit-additivity}
 s_q(aQ+b)=s_q(a)+s_q(b).
\end{equation}
Splitting $\{0,\ldots,N-1\}$ into $M$ complete blocks of length $Q$
and one final block of length $R$, and using
\eqref{eq:digit-additivity}, gives
\begin{equation*}
	\frac1N\sum_{n=0}^{N-1}f(T^{s_q(n)}x) =\frac QN\sum_{a=0}^{M-1}
 \mathcal R_r f\bigl(T^{s_q(a)}x\bigr) +\frac1N\sum_{b=0}^{R-1}
 f\bigl(T^{s_q(M)+s_q(b)}x\bigr).
\end{equation*}

Because $MQ+R=N$ and $|I_f|\leq\norm{f}_\infty$, we may write
\[I_f
=\frac{Q}{N}\sum_{m=0}^{M-1}I_f
 +\frac1N\sum_{b=0}^{R-1}I_f.
 \]
Then it follows by the triangle inequality that
\begin{align}
 \sup_{x\in X}
 \Big|
  \frac1N\sum_{n=0}^{N-1}f(T^{s_q(n)}x)-I_f
 \Big| &\leq
 \sup_{y\in X}|\mathcal R_r f(y)-I_f|
 +\frac{2R}{N}\norm{f}_\infty\notag\\
 &\leq
 \sup_{y\in X}|\mathcal R_r f(y)-I_f|
 +\frac{2Q}{N}\norm{f}_\infty.     \label{eq:block-bound}
\end{align}
For each fixed $r$, let $N\to\infty$ in
\eqref{eq:block-bound}.  We obtain
\begin{equation*}
	\limsup_{N\to\infty}\sup_{x\in X}
 \Big|
  \frac1N\sum_{n=0}^{N-1}f(T^{s_q(n)}x)-I_f
 \Big| \leq
 \sup_{y\in X}|\mathcal R_r f(y)-I_f|.
\end{equation*}

Letting $r\to\infty$ and applying Lemma~\ref{lem:digit-block} proves
the desired uniform convergence in Theorem~\ref{thm:main_sq}.
\end{proof}

\subsection{Proof of Theorem~\ref{thm:one-point}}\label{sec:first-proof}

The proof of Theorem~\ref{thm:one-point} is similar to that of Theorem~\ref{thm:main_sq}.
Define
\[
 \mathcal S_Nf(x):=\frac1N\sum_{n\leq N}
 f\bigl(T^{s_q(a(n))}x\bigr)
 =\sum_{k\geq0}\overline{\eta}_N(k)f(T^{s_q(k)}x).
\]

Fix $r\geq1$, put $Q=q^r$.  For $a\geq0$, set
\[
 \beta_{N,a}:=\frac1Q\sum_{c=0}^{Q-1}\overline{\eta}_N(aQ+c),
\]
and define a probability mass that is constant on each block of length
$Q$ by
\[
 \widetilde{\eta}_{N,Q}(aQ+b):=\beta_{N,a},
 \qquad0\leq b<Q.
\]
We extend both $\overline{\eta}_N$ and $\widetilde{\eta}_{N,Q}$ by zero
on the negative integers.  Every integer $k\geq0$ has a unique representation
$k=aQ+b$ with $a\geq0$ and $0\leq b<Q$.  Therefore, by the definition
of $\beta_{N,a}$ and the triangle inequality,
\begin{align*}
 \norm{\overline{\eta}_N-\widetilde{\eta}_{N,Q}}_{1}
 &=\sum_{a\geq0}\sum_{b=0}^{Q-1}
 \Big|
  \overline{\eta}_N(aQ+b)
  -\frac1Q\sum_{c=0}^{Q-1}\overline{\eta}_N(aQ+c)
 \Big|\notag\\
 &=\sum_{a\geq0}\sum_{b=0}^{Q-1}
 \Big|
  \frac1Q\sum_{c=0}^{Q-1}
  \bigl(\overline{\eta}_N(aQ+b)-\overline{\eta}_N(aQ+c)\bigr)
 \Big|\notag\\
 &\leq\frac1Q\sum_{b,c=0}^{Q-1}
 \sum_{a\geq0}
 \Big|
  \overline{\eta}_N(aQ+b)-\overline{\eta}_N(aQ+c)
 \Big|.
\end{align*}

For any $b,c\in\{0,\ldots,Q-1\}$ and put
$h=c-b$.  Then
\begin{equation*}
	\sum_{a\geq0}
 \Big|
  \overline{\eta}_N(aQ+b)-\overline{\eta}_N(aQ+c)
 \Big| \leq
 \sum_{m\in\Z}
 \Big|
  \overline{\eta}_N(m)-\overline{\eta}_N(m+h)
 \Big|:=V_N(h).
\end{equation*}
Notice that by the telescope inequality, we have $V_N(h) \leq |h|V_N(1)$ for any $h\in \Z$. Then
\begin{equation*}
	\norm{\overline{\eta}_N-\widetilde{\eta}_{N,Q}}_{1} \leq\frac1Q\sum_{b,c=0}^{Q-1}|c-b|V_N(1)=\frac{Q^2-1}{3}V_N(1).
\end{equation*}

Replacing $\overline{\eta}_N$ by $\widetilde{\eta}_{N,Q}$ in $\mathcal S_Nf(x)$, we get that
\[
 \Big|\mathcal S_Nf(x)-\sum_{k\geq0}\widetilde{\eta}_{N,Q}(k)f(T^{s_q(k)}x)\Big| \leq \norm{f}_\infty\sum_{k\geq0}|\overline{\eta}_N(k)-\widetilde{\eta}_{N,Q}(k)| \leq \norm{f}_\infty Q^2V_N(1).
\]

  Using \eqref{eq:digit-additivity},
we have
\begin{align*}
 \sum_{k\geq0}\widetilde{\eta}_{N,Q}(k)f(T^{s_q(k)}x)
 &=\sum_{a\geq0}\beta_{N,a}
 \sum_{b=0}^{Q-1}f\bigl(T^{s_q(a)+s_q(b)}x\bigr)=\sum_{a\geq0}Q\beta_{N,a}\,
 \mathcal R_rf\bigl(T^{s_q(a)}x\bigr).
\end{align*}
The coefficients $Q\beta_{N,a}$ are nonnegative and sum to one over $a\ge0$.
Writing $I_f=\int_Xf\,d\nu$, we get that
\[
\Big|\sum_{k\geq0}\widetilde{\eta}_{N,Q}(k)f(T^{s_q(k)}x) - I_f\Big|\leq \sum_{a\geq0}Q\beta_{N,a}\,
 \mathcal |R_rf\bigl(T^{s_q(a)}x\bigr)-I_f| \leq\sup_{y\in X}|\mathcal R_rf(y)-I_f|.
\]

By the triangle inequality, we obtain
\[
 \sup_{x\in X}|\mathcal S_Nf(x)-I_f|
 \leq\sup_{y\in X}|\mathcal R_rf(y)-I_f|
 +\norm{f}_\infty Q^2V_N(1).
\]
First let $N\to\infty$, and then let $r\to\infty$; applying
Lemma~\ref{lem:digit-block} and the assumption that $V_N(1)=o_{N\to\infty}(1)$ proves the theorem.
\qed

\section{Proof of Theorem~\ref{thm:two-point}}
\label{sec:second-proof}

Set
\[
 \mathcal L_Nf(x):=\mathbb E_{n\leq N}^{\log}
 f\bigl(T^{s_q(\Omega(n))+s_q(\Omega(n+1))}x\bigr).
\]
We first compare $\mathcal L_Nf$ with the square of the
averaging operator $\mathcal S_N$ defined in Section~\ref{sec:first-proof}. We use $\mathcal S_N^2$ to denote $\mathcal S_N \circ\mathcal S_N$.

\begin{lemma}\label{lem:compare-square}
For every $f\in C(X)$,
\[
\sup_{x\in X}
 |\mathcal L_Nf(x)-\mathcal S_N^2f(x)| = O\Big(\frac{(\log\log\log N)^2}{\sqrt{\log\log N}}\Big).
\]
\end{lemma}

\begin{proof}
Let $L=\log\log N$ and $K_N=\lfloor2L\rfloor$.  Define
\[
 \tau_N:=\sum_{k>K_N}\overline{\pi}_N(k)
 =\frac1N\#\{m\leq N:\Omega(m)>K_N\}.
\]
By \eqref{eq:turan-kubilius} and Chebyshev's inequality,
\begin{equation}\label{eq:large-tail}
 \tau_N\ll L^{-1}.
\end{equation}

Applying \eqref{CharamarasRichter2025_thm_A} first with
$a(k)=\mathbf 1_{\{k>K_N\}}$, $b=1$, and then with $a=1$ and
$b(k)=\mathbf 1_{\{k>K_N\}}$, gives
\begin{align*}
 \mathbb E_N^{\log}\mathbf 1_{\{\Omega(n)>K_N\}}
 &=\tau_N+O(\epsilon_N),\\
 \mathbb E_N^{\log}\mathbf 1_{\{\Omega(n+1)>K_N\}}
 &=\tau_N+O(\epsilon_N),
\end{align*}
where $\epsilon_N=\frac1{\sqrt{\log\log N}}$.
It follows that the part of $\mathcal L_Nf(x)$ on which either
$\Omega(n)>K_N$ or $\Omega(n+1)>K_N$ is $O(\epsilon_N)$, uniformly in $x$.

Put
\[
 M_N:=\max_{0\leq k\leq K_N}s_q(k).
\]
An integer $k\leq K_N$ has at most $1+\lfloor\log_qK_N\rfloor$ base-$q$
digits, so
\begin{equation}\label{eq:M-bound}
 M_N=O_q(\log L).
\end{equation}
For $0\leq u\leq M_N$, define
\[
 a_{u,N}(k):=
 \mathbf 1_{\{0\leq k\leq K_N,\ s_q(k)=u\}},
 \qquad
 \alpha_{u,N}:=\BEu{1\leq m\leq N}a_{u,N}(\Omega(m)).
\]
On the truncated range we have the identity
\begin{align}
 &f\bigl(T^{s_q(\Omega(n))+s_q(\Omega(n+1))}x\bigr)
 \mathbf 1_{\{\Omega(n),\Omega(n+1)\leq K_N\}}=\sum_{u,v=0}^{M_N}f(T^{u+v}x)
 a_{u,N}(\Omega(n))a_{v,N}(\Omega(n+1)).               \label{eq:digit-decomp}
\end{align}

Applying \eqref{CharamarasRichter2025_thm_A} to every pair $(a_{u,N},a_{v,N})$, we have
\[
\mathbb E_{n\leq N}^{\log}
a_{u,N}(\Omega(n))a_{v,N}(\Omega(n+1))=\alpha_{u,N}\alpha_{v,N}+O(\varepsilon_N).
\]
Summing over $0\leq u, v\leq M_N$,  \eqref{eq:digit-decomp} gives
\begin{align}
 \mathcal L_Nf(x)
 &=\sum_{u,v=0}^{M_N}f(T^{u+v}x)
 \alpha_{u,N}\alpha_{v,N}+O\bigl(\norm{f}_\infty(M_N+1)^2\epsilon_N\bigr)
 +O(\epsilon_N). \label{eq:log-to-product}
\end{align}
All error terms are uniform in $x$.  Since
$\log\log\log N=\log L$, equation \eqref{eq:M-bound} implies
\[
 (M_N+1)^2\epsilon_N
 \ll_q\frac{(\log L)^2}{\sqrt L}.
\]

Moreover,
\begin{align*}
 \sum_{u,v=0}^{M_N}f(T^{u+v}x)\alpha_{u,N}\alpha_{v,N}
 &=\sum_{0\leq k,\ell\leq K_N}
 \overline{\pi}_N(k)\overline{\pi}_N(\ell)f\bigl(T^{s_q(k)+s_q(\ell)}x\bigr)\\
 &= \sum_{k,\ell\geq0}\overline{\pi}_N(k)\overline{\pi}_N(\ell)
 f\bigl(T^{s_q(k)+s_q(\ell)}x\bigr) + O(\norm{f}_\infty\tau_N)\\
 &=\mathcal S_N^2f(x) + O(\norm{f}_\infty L^{-1})
\end{align*}
by  \eqref{eq:large-tail}. From 
\eqref{eq:log-to-product}, we have
\begin{equation*}
	\mathcal L_Nf(x)= \mathcal S_N^2f(x) + O(\norm{f}_\infty L^{-1}) + O(\norm{f}_\infty(\log L)^2 L^{-1/2}) +O(L^{-1/2}).
\end{equation*}
The conclusion now follows.
\end{proof}

\begin{proof}[Proof of Theorem~\ref{thm:two-point}]
Let $I_f=\int_Xf\,d\nu$.  By Theorem~\ref{thm:one-point},
\[
 \delta_N:=\sup_{y\in X}|\mathcal S_Nf(y)-I_f|\longrightarrow0.
\]
Since $\sum_k\overline{\pi}_N(k)=1$,
\begin{align*}
 |\mathcal S_N^2f(x)-I_f|
 &=\Big|
 \sum_{k\geq0}\overline{\pi}_N(k)
 \bigl(\mathcal S_Nf(T^{s_q(k)}x)-I_f\bigr)
 \Big|\leq\delta_N.
\end{align*}
Lemma~\ref{lem:compare-square} therefore gives
\[
 \lim_{N\to\infty}\sup_{x\in X}|\mathcal L_Nf(x)-I_f|=0.
\]
The theorem follows.
\end{proof}

\section{Some problems}
In this last section, we raise some problems about the variants of Bergelson and Richter's  Theorem \eqref{BergelsonRichter2022}. Motivated by the dynamical results of Bergelson and Richter \cite{BergelsonRichter2022}, Donoso et al. \cite{DLMS2024}, Qi and Zheng \cite{QiZheng2026} on the prime Omega function, we conjecture that the following generalization of  \cref{Chowla_conjecture_BR_form} holds. 

\begin{conjecture}\label{conj_multi_poly}
Let $k\ge1$ be an integer.  Let $P(Y_1,\dots,Y_k)\in \Z[Y_1,\dots,Y_k]$ be a polynomial with coprime coefficients such that there is no integer $d\ge2$ such that the values of $|P(Y_1,\dots,Y_k)|$ at almost all positive integers are $d$th powers. For example, $|P(Y_1,\dots,Y_k)|=Q(Y_1,\dots,Y_k)^d$ for some $d\ge2$ is excluded from consideration. Let $(X,\nu, T)$ be uniquely ergodic, then
	\begin{equation}\label{conj_multi_poly_eqn}
		\lim_{N\to\infty} \frac1{N^k}\sum_{1\leq n_1,\dots, n_k\leq N} f(T^{\Omega(|P(n_1,\dots,n_k)|)}x) = \int_X f\,d\nu
	\end{equation}
holds for any $x\in X$ and $f\in C(X)$.	
\end{conjecture}
That is, the obstruction that fails \eqref{conj_multi_poly_eqn} occurs mainly at the power values of $|P(Y_1,\dots,Y_k)|$.

 Let $p(n)$ be the unrestricted partition function. In \cite{ParkinShanks1967}, Parkin and Shanks conjectured that
\[
\lim_{N\to\infty}\frac1N\#\set{n\leq N: 2\mid p(n)}=\frac12 \quad\text{  and }\quad \lim_{N\to\infty}\frac1N\#\set{n\leq N: 2\nmid p(n)}=\frac12.
\]
Similarly, we pose the following conjecture.
\begin{conjecture}
We have
\[
\lim_{N\to\infty}\frac1N\#\set{n\leq N: 2\mid \Omega(p(n))}=\frac12 \quad\text{  and }\quad \lim_{N\to\infty}\frac1N\#\set{n\leq N: 2\nmid \Omega(p(n))}=\frac12.
\]
\end{conjecture}

Let $\varphi(n)$ be the Euler's totient function. Let $\sigma(n)$ be the sum of positive divisors of $n$. Let $\tau(n)$ be the Ramanujan tau function.  Let $$r_4(n)=|\set{(m_1,\dots,m_4)\in\Z^4: n=m_1^2+\cdots+m_4^2}|.$$
Let $\mathcal{P}$ be the set of all primes and let  $$R_m(n)=|\set{(p_1,\dots, p_m)\in \mathcal{P}^m: n=p_1+\cdots+ p_m}|,$$ 
where $n$ is even if $m$ is even, and $n$ is odd if $m$ is odd. Let $a(n)$ be $\varphi(n)$,  $\sigma(n)$, $|\tau(n)|$, $r_4(n)$, $R_m(n)$, $m\ge2$, $p(n)$ or $n+\varphi(n)$. A natural question is whether \eqref{BergelsonRichter2022} still holds if $\Omega(n)$ is replaced by $\Omega(a(n))$. Moreover, 
for for distinct nonnegative integers  $h_1, \ldots, h_k$, one may ask whether the following analogue of Chowla's conjecture holds:
		\begin{equation*}
			\lim_{N\to\infty}\frac{1}{N} \sum_{n\leqslant N} \lambda(a(n+h_1))\ldots\lambda(a(n+h_k)) = 0.
		\end{equation*}
		 If these conjectures are true, they would reveal rich independence phenomena between multiplicative and additive structures.

\section*{Acknowledgments}

AI tools were used to explore some of the proof strategies and calculations. The authors take full responsibility for human verifications and the content.

\end{document}